\documentclass[oneside, a4paper,11pt]{amsart}

\usepackage[utf8]{inputenc}
 \usepackage[T2A,T1]{fontenc}
 
 \DeclareSymbolFont{cyrillic}{T2A}{cmr}{m}{n}
 \DeclareMathSymbol{\zemlja}{\mathalpha}{cyrillic}{199}
 \DeclareMathSymbol{\sha}{\mathalpha}{cyrillic}{216}
 \DeclareMathSymbol{\cheiri}{\mathalpha}{cyrillic}{149}
 \DeclareMathSymbol{\kako}{\mathalpha}{cyrillic}{138}
  \DeclareMathSymbol{\yer}{\mathalpha}{cyrillic}{218}

\usepackage{geometry}

\usepackage{amsmath}
\usepackage{amsfonts}
\usepackage{amsthm}
\usepackage{amssymb}
\usepackage{graphicx}
\usepackage{tikz}
\usepackage{multirow}

\usepackage{aliascnt}
\usepackage{hyperref}
\hypersetup{ 
    colorlinks=false,
    pdfborder={0 0 0},}

\usepackage{array} 

\newcommand{\Ker}{\mbox{Ker}}
\renewcommand{\mod}{\mbox{mod}}
\newcommand{\Mod}{\mbox{Mod}}
\newcommand{\End}{\mbox{End}}

\newcommand{\rank}{\mathrm{rank}}
\renewcommand{\dim}{\mathrm{dim}}
\newcommand{\Vect}{\mathrm{Vect}}
\DeclareMathOperator{\ord}{ord}
 
\newenvironment{narrow}[2]{
 \begin{list}{}{%
  \setlength{\topsep}{0pt}%
  \setlength{\leftmargin}{#1}%
  \setlength{\rightmargin}{#2}%
  \setlength{\listparindent}{\parindent}%
  \setlength{\itemindent}{\parindent}%
  \setlength{\parsep}{\parskip}%
 }%
\item[]}{\end{list}} 

\usepackage{mathtools}

\newcommand{\id }{\mathrm{id}}
\newcommand{\B}{\mathcal{B}} 	
\newcommand{\g}{\mathfrak{g}} 	
\renewcommand{\sl}{\mathfrak{sl}}	
\newcommand{\so}{\mathfrak{so}}	
\newcommand{\Z}{\mathbb{Z}}  	
\newcommand{\N}{\mathbb{N}}  	
\renewcommand{\C}{\mathbb{C}}  	
\newcommand{\Q}{\mathbb{Q}}  	

\newcommand{\V}{\mathcal{V}}	

\theoremstyle{plain}
\newtheorem{theorem}{Theorem}[section]

\newtheorem*{acknowledgementX}{Acknowledgement}

\newtheorem{conjecture}[theorem]{Conjecture}
\newtheorem{corollary}[theorem]{Corollary}

\newtheorem{lemma}[theorem]{Lemma}

\newtheorem{definition}[theorem]{Definition}

\newtheorem{question}[theorem]{Question}
\newtheorem{remark}[theorem]{Remark}

\theoremstyle{remark}
\newtheorem{example}[theorem]{Example}

\usepackage{enumerate}

\renewcommand{\exp}[1]{{\mathrm e}^{#1}}

\renewcommand{\i}{{\mathrm i}}

\newcommand{\Y}{{\mathrm Y}}
\newcommand{\resY}{\mathrm{resY}}

\newcommand{\W}{{\mathcal W}}
\renewcommand{\L}{{\mathcal L}}

\newcommand{\oshort}{\ominus}
\newcommand{\olong}{\oplus}
\newcommand{\ocartan}

\catcode`,\active

\catcode`\,12

\newcommand{\md}{\text{-}}

\newcommand{\res}[1]{\mathrm{res}\left(#1\right)}
\newcommand{\zem}{\zemlja}

\newcommand{\SF}{\mathcal{SF}}

\newcommand{\hamburger}[4] 
{
  \thispagestyle{empty}
  \vspace*{-2cm}
  \begin{flushright}
    ZMP-HH / #2 \\
    Hamburger Beitr{\"a}ge zur Mathematik Nr. #3 \\
    #4\\
  \end{flushright}
  \vspace{0.5cm}
  \begin{center}
    \Large \bf
    #1
  \end{center}
  \vspace{0.2cm}
  \begin{center}	
    Ilaria Flandoli, Simon Lentner \\
    Algebra and Number Theory, Center of Mathematical Physics \\
    University Hamburg, 
    Bundesstra{\ss}e 55, D-20146 Hamburg \\
    \texttt{simon.lentner@uni-hamburg.de}
  \end{center}
  \vspace{-0.8cm}
}

\begin{document}

\hamburger{Logarithmic conformal field theories of type $B_n,\ell=4$ \\ and symplectic fermions}{17-18}{663}{June 2017}

\begin{abstract}
  There are important conjectures about logarithmic conformal field theories (LCFT), which are constructed as kernel of screening operators acting on the vertex algebra of the  rescaled root lattice of a finite-dimensional semisimple complex Lie algebra. In particular their representation theory should be equivalent to the representation theory of an associated small quantum group. 
  
  This article solves the case of the rescaled root lattice $B_n/\sqrt{2}$ as a first working example beyond $A_1/\sqrt{p}$. We discuss the kernel of short screening operators, its  representations and graded characters. Our main result is that this vertex algebra is isomorphic to a well-known example: The even part of $n$ pairs of symplectic fermions.
  
  In the screening operator approach this vertex algebra appears as an extension of the vertex algebra associated to rescaled $A_1^n$, which are $n$ copies of the even part of one pair. The new long screenings give the global $C_n$-symmetry. The extension is due to a degeneracy in this particular case: Rescaled long roots still have even integer norm.
 
  The associated quantum group of divided powers has similar degeneracies \cite{LentCCM}: It contains the small quantum group of type $A_1^n$ and the Lie algebra $C_n$. Recent results \cite{FGR17b} on symplectic fermions suggest finally the conjectured category equivalence to this quantum group. 
  We also study the other degenerate cases of a quantum group, giving extensions of LCFT's of type $D_n,D_4,A_2$ with larger global symmetry $B_n,F_4,G_2$. 
\end{abstract}
\title[]{}
\maketitle

\newpage

\thispagestyle{empty}
\setcounter{tocdepth}{2}
\tableofcontents

\newpage

\section{Introduction}

An important set of conjectures \ref{conj_Main} is concerned with the construction of vertex algebras $\W$, underlying logarithmic conformal field theories, which should have the same representation theory as a small quantum group $u_q(\g)$. The goal of this article is to treat the examples $\g=B_n,\ell=4$ in this regard as a nice working example beyond the known case $A_1$. As conclusion we are able to prove that $\W$ is isomorphic to the vertex algebra of $n$ pairs of symplectic fermions $\V_{\SF_n^{even}}$. Moreover from recent work \cite{FGR17b} on symplectic fermions (if one of their conjectures holds) we can deduce that the representation category is equivalent to representations of a quasi-Hopf algebra $\tilde{u}$ closely related to the quantum group $u_q(A_1^n)$, which is the expected quantum group (section \ref{sec_quantumgroup}). In particular up to this issue all conjectures \ref{conj_Main} hold in this example.\\

More precisely, let $\g$ be a complex finite-dimensional semisimple Lie algebra and $q$ a primitive $\ell$th root of unity. Then the \emph{small quantum group} $u_q(\g)$ is a Hopf algebra defined by Lusztig in \cite{Lusz90a}: It is a finite-dimensional quotient of a $q$-deformation of the universal enveloping algebra $U(\g)$, specialized in a particular way to an $\ell$-th root of unity $q$. Since $u_q(\g)$ is a Hopf algebra, its category of finite-dimensional representations is naturally endowed with a monoidal structure $\otimes_\C$ and dualities, much like representations of the Lie algebra $\g$ can be tensored and dualized. This category is non-semisimple and related to the representations of $\g$ in finite characteristic $\ell$ (if prime) and to the representations of the affine Lie algebra at level $-\ell-h^\vee$. For odd $\ell$, the category can be endowed with a nondegenerate braiding and thus becomes a \emph{modular tensor category}. For even $\ell$ (as in this article) this is still morally true, but in detail one either has to extend the coradical, see \cite{LO16}, or introduce a nontrivial associator as in \cite{GR}. One application of this structure is the construction of topological invariants of knots and $3$-manifolds \cite{Lyu}.\\

Another source for modular tensor categories are representation of \emph{vertex algebras}. These are algebraic structures with an extra layer of analysis, in that the multiplication $a\otimes b\mapsto \Y(a)b$ is a Laurent series in a formal parameter $z$, and it is compatible with an action of the Virasoro algebra. Vertex algebras encode an essential part of conformal quantum field theories.
A vertex algebra has an associated representation theory and under certain finiteness-conditions one can show that this category is always a braided tensor category \cite{HLZ10}, consisting of infinite-dimensional graded vector spaces. If this category is semisimple, it is automatically a modular tensor category (this is conjecturally true regardless of semisimplicity) and the graded dimensions $\sum_n \dim(V_n)t^n$ of the irreducible modules piece together to a vector-valued modular form (in the nonsemisimple case additional pseudo-characters are required). A main example in what follows is the \emph{lattice vertex algebra} $\V_{\Lambda}$ associated to an even integral lattice $\Lambda$; physically this is a free boson on a torus. In this case the representation category is a semisimple modular tensor equivalent to the category of $\Lambda^*/\Lambda$-graded vector spaces, with associator $\omega$ and braiding $\sigma$ derived from the quadratic form $e^{\pi\i(\lambda,\lambda)}$. The graded dimension of an irreducible  module $\V_{[\lambda]}$ associated to a coset $[\lambda]\in \Lambda^*/\Lambda$ is essentially the Jacobi $\Theta$-function associated to this lattice coset.\\

In this article we study examples of vertex algebras, which fulfill the finiteness conditions, but have non-semisimple representation theory with nondiagonalizable $L_0$-action of the Virasoro algebra. We call such a vertex algebra \emph{logarithmic} opposed to \emph{rational} in the semisimple case; the name comes from the singularites in the associated analytic functions. This non-semisimple theory is much less developed than the semisimple theory and only few examples are known. The best understood are the triplet algebra $\W_{p}$ \cite{Kausch91,FFHST02,AM08}, its generalization $\W_{p,q}$ \cite{FGST06c,GRW09,AM11}, and the algebra of $n$ pairs of symplectic fermions $\V_{\SF_n^{even}}$ \cite{Kausch,Abe07,Runkel12,DR16}.

An intriguing construction for vertex algebras are \emph{free-field-realizations}, starting with \cite{Wak86,FF,Fel89}: One begins with a lattice vertex algebra $\V_\Lambda$ and constructs \emph{screening operators} acting on the vertex algebra and its representations. Then the kernel of the screening operators give usually a much more complicated vertex algebra $\W$. In particular a famous set of conjectures \ref{conj_Main} is concerned with $\Lambda$ the root-lattice of a Lie algebra $\g$ rescaled by some number $\ell$, and $\W$ the kernel of short screenings. Then the non-semisimple representation category of $\W$ should be equivalent to the representation category of the respective quantum group $u_q(\g)$. 

Much work has been done on the case $\g=\sl_2$, which is now solved as an abelian category \cite{FFHST02,FGSTsl2,AM08,TN,TW}. In particular for $\ell=2p=4$ the vertex algebra $\W$ is isomorphic to the known LCFT $\V_{\SF_1^{even}}$ and for $\ell=2p$ is isomorphic to the known triplet algebra $\W_p$; in this case \cite{GR,CGR17}  construct a quasi-Hopf algebra variant of the quantum group $u_q(\sl_2)$, such that there is even an equivalence of braided monoidal categories. In \cite{FT} the program is discussed for arbitrary simply-laced Lie algebras $\g$, including conjectural formulae for the graded dimensions in terms of false $\Theta$-functions. In \cite{ST12,ST13,S14} the program has been extended to Nichols algebras, stressing the general appearance of quantum symmetrizers. In \cite{LentQGNA} the second author has proven in general that the screening operators give an action of the small quantum group $u_q(\g)^+$ (and more generally of a diagonal Nichols algebra) on the vertex algebra.\\

In Section 2 we thoroughly discuss the screening charge method with minor own additions to cover the case of $\g$ non-simply-laced. The important output is the lattice $\Lambda^\olong$ on which the free-field realization $\V_{\Lambda^\olong}$ is built, with a basis of \emph{long screening momenta} $\alpha_i^\olong$, and a finer lattice ${\Lambda^\oshort}$ with a basis of \emph{short screening momenta} $\alpha_i^\oshort$, together with a choice of an action of the Virasoro algebra. Then \emph{long} and \emph{short  screening operators} $\zem_{\alpha_i^\olong},\zem_{\alpha_i^\oshort}$ are defined by taking formal residues of the power series associated to the vertex operator of certain elements $\Y(\exp{\phi_{\alpha_i^\olong}}),\Y(\exp{\phi_{\alpha_i^\oshort}})$. The second vertex operator consists of fractional $z$-powers, which makes everything much more involved and interesting. The choices are such that by \cite{LentQGNA} the short screening operators give an action of the small quantum group. All long screenings operators preserve the Virasoro action, as do suitable powers of short screening operators; this is the idea underlying the Felder complex \cite{Fel89} in rank $1$.\\

In Section 3 we formulate the interesting conjectures about the vertex algebra $\W\subset \V_{\Lambda^\olong}$ that is obtained by taking the kernel of all short screenings $\zem_{\alpha_i^\oshort}$.\\

In Sections 4, 5, 6 we thoroughly discuss the known example $A_1=B_1,\ell=4$ and the new examples $B_2,\ell=4$ before the general case $B_n,\ell=4$. In each case we describe the involved lattices, the kernel of the short screenings $\W$, the representations $\V_{[\lambda]}$ of the lattice vertex algebra $\V_{\Lambda}$ and then their decomposition when we restrict them to $\W$. These examples are interesting because they are the first beyond $A_1$, and also because they are non-simply-laced, which e.g. implies that the lattice $\Lambda^\oshort$ is a rescaled $B_n$-lattice while the lattice $\Lambda^\olong$ is a rescaled $C_n$-lattice of the dual Lie algebra. 

In addition because the value of $\ell$ is small compared to the root lengths of long roots, these examples have an interesting and simplifying degeneracy: Because short screenings associated to long simple roots in $B_n$ are equal to long screenings of the short coroots in $C_n$, we actually only have to work with short screenings associated to short roots of $B_n$, which is a subsystem of type $A_1^n$. It is a nice educational feature of this example that these three lattices serving three different purposes in the construction are actually all different.\\

In Section 7 we prove that the kernel of the screening $\W$ is isomorphic to a known LCFT: The vertex algebra of $n$ pairs of symplectic fermions $\V_{\SF_n^{even}}$. First we describe the vertex algebra $\V_{\SF_n^{even}}$ and the super vertex algebra $\V_{\SF_n}$ and their representations. Then we match the graded dimensions of the irreducible modules to those of $\W$ constructed in the previous section; this was a very promising check early in our work. We finally prove the vertex algebra isomorphism of $\W$ for $B_n,\ell=4$ to $\V_{\SF_n^{even}}$.

In view of the fact that this is well-known for $A_1=B_1,\ell=4$, and that we can only take short screenings $A_1^n$, and that the super vertex algebras $\V_{\SF_n}=(V_{\SF_1})^n$, this result is not very surprising and not hard to prove. Our $\W$ contains a sub vertex algebra equivalent to $(\V_{\SF_1^{even}})^n$ and the only thing to prove is that the larger lattice $B_n\supset A_1^n$ leads to an extension of this lattice algebra to $(\V_{\SF_1})^{n,even}=\V_{\SF_n^{even}}$. At the same time the Lie algebra generated by the long screenings extends from $A_1^n=\sl_2^n$ to $C_n=\mathfrak{sp}_{2n}$, which is the right global symmetry for $\V_{\SF_n},\V_{\SF_n^{even}}$.\\

In Section 8 we turn to the quantum group side for $B_n,\ell=4$, where the picture is very nicely matching: Because the order of $\ell$ is small compared to the root lengths in this non-simply-laced Lie algebra, the usual theory does not quite hold: If $\ell$ has no common divisors with root lengths, then by \cite{Lusz90a} the infinite-dimensional Lusztig quantum group of divided powers $U_q^\L(\g)$, the specialization of the generic deformation $U_q(\g)$ to a root of unity, decomposes into the small quantum group $u_q(\g)$ and $U(\g)$.

The case of arbitrary $\ell$ is more involved and treated in the second author work \cite{LentCCM}: For $B_n,\ell=4$ it turns out that the Lusztig quantum group decomposes into the small quantum group $u_q(A_1^n)$, associated to only the short root vectors, and $U(\g^\vee)=U(C_n)$ acting on it by adjoint action.\\

As conclusion we are able to deduce from recent work \cite{FGR17b} on symplectic fermions (if one of their conjectures holds) that the representation category is equivalent to representations of a quasi-Hopf algebra $\tilde{u}$ closely related to the quantum group $u_q(A_1^n)$, which is what we expect from the quantum group side.\\

In Section 9 we discuss briefly the other degenerate cases $C_n,F_4,\ell=4$ and $G_2,\ell=6$, which should again be LCFT's that are extensions of LCFT's of type $D_n,D_4,A_2$ with appropriate central charge. We would be interested to know if these examples are already known, in order to compare the structure and to increase the number of examples for which the screening charge method is known to work.

\section{Setting and general results}

The following Hopf algebra description of generalized vertex algebras associated to (non-integral) lattices and screening charge operators has been developed in \cite{Len07,LentQGNA}. We repeat the relevant definitions and give examples, but for further reading we refer the reader to the second paper.

\subsection{The lattice vertex algebra}
Let $\Lambda \subset \mathbb{C} ^{\rank}$ be a lattice with basis $\left\{\alpha_1, \ldots, \alpha_{\rank}\right\}$ and inner product $(\;,\;): \Lambda \times \Lambda \rightarrow \frac{1}{N} \mathbb{Z}$. 

\begin{definition}
  $\V_\Lambda$ is the commutative, cocommutative, infinite-dimensional $\mathbb{N}_0$-graded Hopf algebra  generated by the formal symbols  
\begin{displaymath} e^{{\phi}_\beta}, \qquad \qquad {\partial^{1+k}\phi}_\alpha,\qquad  \alpha, \beta\in \Lambda,\;k\in\N_0
\end{displaymath} 
They are subject to the algebra relations  
\begin{align*}
e^{{\phi}_\alpha}e^{{\phi}_\beta}
&= e^{{\phi}_{\alpha + \beta}}  &\alpha, \beta \in \Lambda\\
{\partial^{1+k}\phi}_{a\alpha + b\beta} 
&= a{\partial^{1+k}\phi}_\alpha + b {\partial^{1+k}\phi}_\alpha &a, b \in \mathbb{Z}
\end{align*}
As a coalgebra the symbols ${\partial^{1+k}\phi}_\alpha$ are primitive and the $e^{{\phi}_\beta}$ are grouplike:
\begin{align*}
\Delta{\partial^{1+k}\phi}_\alpha 
&= 1\otimes{\partial^{1+k}\phi}_\alpha + {\partial^{1+k}\phi}_\alpha \otimes 1 \\
\Delta e^{{\phi}_\beta}
&= e^{{\phi}_\beta} \otimes e^{{\phi}_\beta}
\end{align*} 
In the $\N_0$ grading the $e^{{\phi}_\beta}$ have degree $0$ and the ${\partial^{1+k}\phi}_\alpha$ have degree $k$. We call $\left|u\right|$ the $\mathbb{N}_0$-degree of a differential polynomial $u$.\\

\noindent
An arbitrary element in $\V_\Lambda$ is of the form $ue^{{\phi}_\lambda}$ with $\beta \in \Lambda$ and $u$ differential polynomial.
\end{definition}

Additional structure on $\V_\Lambda$ is a Hopf algebra derivation  $\partial$, raising the $\N_0$-degree by $1$.
  \begin{align*}
  \partial.\exp{\phi_\alpha}
  &:=\partial\phi_\alpha\cdot\exp{\phi_\alpha}\\
  \partial.\partial^k\phi_\alpha
  &:=\partial^{k+1}\phi_\alpha
  \end{align*}
and a Hopf pairing on $\V_\Lambda$ with values in the ring of Laurent polynomials $R=\C[z^{1/N},z^{-1/N}]$, equivariant with respect to the action of $\partial$ on $R$ by $-\frac{\partial}{\partial z}$
  \begin{align*}
   \langle \exp{\phi_\alpha},\exp{\phi_\beta} \rangle
   &=z^{(\alpha,\beta)}\\
   \langle \exp{\phi_\alpha},\partial \phi_\beta \rangle
   &=-(\alpha,\beta)z^{-1}\\
   \langle \partial \phi_\alpha, \exp{\phi_\beta} \rangle
   &=(\alpha,\beta)z^{-1}\\
   \langle \partial \phi_\alpha, \partial\phi_\beta \rangle
   &=(\alpha,\beta)z^{-2}
  \end{align*}

As defined in \cite{Len07,LentQGNA}, these structures give rise to a construction-scheme for \emph{generalized vertex operator}. Relevant properties generalizing locality and associativity can be proven in general and it is also proven in general that for integer $z$-powers $R=\C[z,z^{-1}]$  the resulting structure is always a vertex algebra in the familiar sense. It would be desirable to link these results to generalized vertex algebras in the sense of \cite{DL93}.

In the present article we  restrict ourselves to the fractional lattice VOA with $(,):\Lambda\times \Lambda\to \frac{1}{N}\Z$ defined as follows:

\begin{definition}
  The \emph{Vertex operator} $\Y$ is defined as 
  \begin{align*} 
  \Y : \V_\Lambda &\rightarrow \End (\V_\Lambda )[[z^{\frac{1}{N}} ,z^{-\frac{1}{N}}]]\\ 
  a &\longmapsto \left(b\mapsto \sum \limits_{k \geq 0} \left\langle a^{(2)}, b^{(2)}\right\rangle \cdot b^{(1)} \cdot \frac{z^k}{k!} \partial^k .a^{(1)}\right) 
  \end{align*}
\end{definition}
\begin{remark}
  More generally, for some function $\kappa:\Lambda\times \Lambda\to \C^\times$ we can twist the relation of the groupring $e^{{\phi}_\alpha}e^{{\phi}_\beta}
  = \kappa(\alpha,\beta) e^{{\phi}_{\alpha + \beta}}$, then $\V_\Lambda$ is a comodule algebra over $\V_\Lambda$ if $\kappa$ is a $2$-cocycle resp. over the coquasi Hopf algebra $\V_\Lambda^{d\kappa}$. 
  
  This is frequently done for the lattice algebra to remove naturally appearing signs. It is similarly done in \cite{FT} to remove naturally appearing anticommutators in the Lie algebra. In the following this choice does not matter. 
\end{remark}
\begin{theorem}
If the lattice $\Lambda$ is an even resp odd integer lattice, then this defines the usual structure of a lattice vertex algebra resp. super vertex algebra. 

If $\Lambda$ is a fractional lattice, then the locality in $\V_\Lambda$ is replaced by much more complicated Nichols algebra relations, see \cite{LentQGNA} Section 4; it is these relations that make the screening operators in the next section interesting. 
\end{theorem}

We now discuss the representation theory of the vertex algebra $\V_\Lambda$, say for $\Lambda$ an even integer lattice. The following results are well-known in literature, for the abelian category see \cite{Dong93}, the general tensor product \cite{HLZ10} can be constructed in these case by considering the generalized vertex algebra $\V_{\Lambda^*}$, and the associator and braiding follow uniquely from the quadratic form\footnote{Thanks to T. Gannon for pointing this out to me.}:

\begin{theorem}
   Assume $\Lambda$ an even integer lattice, then $\V_\Lambda\md\Mod$ is a semisimple modular tensor category. The simple objects are parametrized by classes $[\lambda]\in \Lambda^*/\Lambda$ for the dual lattice $\Lambda^*$ and they can be explicitly realized inside the generalized vertex algebra $V_{\Lambda^*}\in \V_\Lambda\md\Mod$ with the generalized vertex operator above. It decomposes as a $\V_{\Lambda}$-module
   $$\V_{\Lambda^*} = \bigoplus_{[\lambda] \in \Lambda^*/\Lambda}  \V_{[\lambda]}
   \qquad V_{[\lambda]}:=\{u \exp{\phi_\beta}\mid \beta\in[\lambda]\}$$
   and we call $\V_\Lambda=\V_{[0]}$ the \emph{vacuum module} or \emph{regular representation}.\\
   
   As a braided tensor category $\V_\Lambda\md\Mod$ with a Virasoro action given by $T^Q$ is equivalent to the modular tensor category of graded vector-spaces $\Vect_{\Lambda^*/\Lambda}$ with associator and braiding given by any abelian $3$-cocycle $(\omega,\sigma)$ associated to the quadratic form 
   $$F([\lambda])=e^{\pi\i\;(\lambda-Q,\lambda-Q)-(Q,Q)}$$
   Then even integrality condition ensures $F$ is a well-defined function on $\Lambda^*/\Lambda$, furthmore it does not change under adding some $\kappa$ as in the previous remark. The nondegeneracy of the double-braiding $e^{\pi\i\;2(\lambda,\mu)}$ clearly holds by construction.
\end{theorem}
 Note however that the apparent braiding candidate $\sigma(\lambda,\mu)=e^{\pi\i\;(\lambda-Q,\mu-Q)-(Q,Q)}$ is not well-defined on classes and requires the choice of representatives, which causes a suitable associator $\omega$

\subsection{Screening operators}

From the vertex operator, we can produce linear endomorphisms of $\V_\Lambda$ as follows:
\begin{definition}
For a given state $a \in \V_\Lambda$ and $m \in \frac{1}{N} \mathbb{Z}$ the \emph{mode operator} 
$$\Y(a)_m : \V_\Lambda \rightarrow \V_\Lambda$$
is defined as the $z^m$-coefficient of the vertex operator 
$$\Y(a)b= \sum_m z^m\cdot{\Y(a)_mb }$$
Hence in explicit Hopf algebra terms it is
$$ \qquad b \longmapsto \sum\limits_{k \geq 0}{\left\langle a^{(2)}, b^{(2)}\right\rangle}_{-k+m} b^{(1)}\frac{1}{k!}\partial^k.a^{(1)}$$
where ${\langle a, b\rangle}_m$ denotes the $z^m$-coefficient and $a^{(1)}\otimes a^{(2)}$ the coproduct in $\V$.
\end{definition}
Moreover we can define on the level of power series
\begin{definition}
 For a given state $a \in \V_\Lambda$ and the \emph{residue operator} $\resY(a)$ is defined as
  $$
  \resY{(a)}\;b=\begin{dcases*}
		\hspace{4cm}\Y(a)_{-1} 
		& for $m\in\Z$\\
		\frac{e^{2\pi\i m}-1}{2\pi\mathrm \i}
		\sum_{k\in\Z}
		\frac{1}{m+k+1}\Y(a)_{m+k}
		& for $m\not\in\Z$              
		\end{dcases*}
		$$ 
where the residue is defined as a formal residue of fractional polynomials
$$\res{z^m}
  \;=\; 
  \frac{1}{2\pi \mathrm i}\oint_{\mathrm{S}^1} z^m\;\mathrm{d}z
  \;:=\;
  \begin{cases}
		0 & m\in\Z \backslash\{-1\}\\
		1 & m= -1\\
		\frac{1}{2\pi\mathrm i\;(m+1)}
		\left(e^{2\pi\mathrm{i}\;(m+1)}-1\right),\quad
		& m\not\in\Z \\             
		\end{cases}$$ 
For integer lattices $\resY(a)=\Y(a)_{-1}$, but for fractional lattices it can be an infinite linear combination. In this article we will only work with products of residue operators that have only finitely many nonzero terms, see Corollary \ref{cor_Weylaction}.		
\end{definition}
Geometrically the residue defined this way is the integral along the unique lift of a circle with given radius to the multivalued covering on which the polynomial with fractional exponents is defined. 

Operators defined this way are automatically defined on any vertex algebra module. In the usual integer case (i.e. $\Y(\partial\phi_\beta)v$ consist only of integer $z$-powers) the endomorphism $\resY(a)=\Y(a)_{-1}$ has several nice properties, in particular one can prove from the OPE associativity a derivational property
$$\Y(a)_{-1}\left(\Y(b)_m c\right)=\Y(\Y(a)_{-1}b)_m c\pm \Y(b)_m\left( \Y(a)_{-1}c\right)$$

\begin{example}
  Let $a=\partial\phi_\alpha$, then the residue operator $\resY(\partial\phi_\alpha)=\Y(\partial\phi_\beta)_{-1}$ is always integer (i.e. $\Y(\partial\phi_\beta)v$ consist only of integer $z$-powers) and this particular residue operator simply gives the $\Lambda$-grading
  $$\resY(\partial\phi_\alpha)\;u\exp{\phi_\lambda}=(\alpha,\lambda)\cdot u\exp{\phi_\lambda}$$
\end{example}

\noindent
The crucial definition in our efforts is:
\begin{definition}\label{def_screening} For $\alpha \in \Lambda^*$ we define the \emph{screening charge operator} $\zem_\alpha$ as
\[\zem_\alpha v := \resY(e^{{\phi}_\alpha})
\]
By definitions, in the integer case $(\alpha, \beta) \in \mathbb{Z}$ this simplifies to 
\begin{equation} \label{defscr}
\zem_\alpha ue^{{\phi}_\beta} = \sum\limits_{k \geq 0}{\langle e^{{\phi}_\alpha}, u^{(0)}\rangle}_{-k-1-(\alpha, \beta)} u^{(1)}e^{{\phi}_\beta}\frac{1}{k!}\partial^k.e^{{\phi}_\alpha}
\end{equation}
\end{definition}

\noindent
We readily observe that the screening operators shift the $\Lambda$-grading 
$$\zem_\alpha: \V_{\lambda} \rightarrow \V_{\lambda + \alpha}
\qquad \V_{\lambda}=\{u\exp{\phi_\lambda}\}$$
and is thus a linear maps between the $\V_\Lambda$-modules
$$\zem_\alpha: \V_{[\lambda]} \rightarrow \V_{[\lambda + \alpha]}$$
In particular for $\alpha \in \Lambda$ they map each module to itself.\\

The screening operators $\zem_\alpha$ are an old and very useful construction in CFT, see for example \cite{DF84,Wak86,FF,Fel89}. In the latter references they have been in fact used for $\alpha$ from a fractional lattice containing an integer lattice $\Lambda$, but here most of the usual properties do not hold on. Our new result \cite{LentQGNA} proves the long-conjectured relations of screenings:
\begin{theorem}\label{teo6.1}[\cite{LentQGNA} Thm. 6.1]
Let $\Lambda$ be a positive-definite lattice and $\left\lbrace \alpha_1, \ldots , \alpha_{n}\right\rbrace $ be a fixed base that fulfils $\left|\alpha_i\right| \leq 1 $.
Then the endomorphisms $\zem_{\alpha_i}: = \resY( e^{{\phi}_{\alpha_i}})$ on the fractional lattice VOA $\V_\Lambda$ constitute an action of the diagonal Nichols algebra generated by the $\zem_{\alpha_i}$ with braiding matrix 
\[ q_{ij} = e^{\pi i (\alpha_i, \alpha_j)}.\]
\end{theorem}
The theorem proves that the screening operators associated to a small enough lattice obey Nichols algebra relations, so we have an action of the Nichols algebra by screening operators. In the present context we will choose lattices such that this Nichols algebra is the positive part of the quantum group $u_q(\g)^+$.

\begin{corollary}\label{cor_NicholsRelations}
In the present article, the following two special cases will be used:
\begin{itemize}
	\item $\left(\alpha, \alpha\right)$ odd integer  $\Rightarrow (\zem_{\alpha})^2 = 0$.
	\item $\left(\alpha, \beta\right)$ even/odd integer  $\Rightarrow [\zem_{\alpha},\zem_{\beta}]_\pm=0$.
\end{itemize}
if the condition $|\alpha|,|\beta| \leq 1$ holds.
\end{corollary}

\subsection{Virasoro action}
Another well-established application of mode operators $Y(a)_m$ is the definition of an action of the Virasoro algebra on $\V_\Lambda$. Our ultimate goal in the next two sections is to choose such a Virasoro action such that a set of screening operators are Virasoro homomorphisms. 

\begin{definition} The Witt algebra is a Lie algebra generated by the vector fields \[ L_n := -z^{n+1} \frac{\partial}{\partial z}\] with $n \in \mathbb{Z}$. The Lie bracket of two vector fields is given by \[[L_m , L_n] = (m - n)L_{m+n}\]
\end{definition}
This is the Lie algebra of the group of diffeomorphisms of the circle.
\begin{definition}
The Virasoro algebra $Vir_c$ is the non-trivial central extension of the Witt Lie algebra by a central element $C$. One usually prescribes that $C$ acts by a fixed scalar $c$, the \emph{central charge}. More precisely the Virasoro algebra is generated by the operators $ L_n$ indexed by an integer  $n \in \mathbb{Z}$ that fulfil the commutation relations
\begin{align*} 
&[L_m , L_n] = (m - n)L_{m+n} + \frac{C}{12}(m^3 - m) \delta_{m+n,0} 
\qquad [L_n , C] = 0  \qquad \qquad \forall n \in \mathbb{Z}.
\end{align*}
\end{definition}

On any VOA we may try to choose an element $T \in \V_\Lambda$ (Energy-Stress-Tensor) such that the mode operators $L_n:=\Y(T)_{-2-n}$ fulfill the relations of the Virasoro algebra.
On any lattice VOA $\V_\Lambda$ we follow Feigin and Fuchs \cite{FF} and define an entire family of such structures parametrized by a parameter $Q$ as follows:
\begin{lemma}\label{lm_Q}
For any $Q\in\Lambda\otimes_\Z\C$ and any dual basis $(\alpha_i, {\alpha_j}^*) = \delta_{ij}$ of $\Lambda$ we define 
$$T := \frac{1}{2} \sum_i \partial\phi_{\alpha_i} \partial\phi_{\alpha_i*} + \sum_i Q_i \partial^2 \phi_{\alpha_i}
\quad \in \quad \V_{\Lambda}$$
Then the mode operators $L_n := \Y(T)_{-2-n}$ constitute an action of the Virasoro algebra $Vir_c$ with central charge $c = \rank - 12 (Q, Q)$ on $\V_\Lambda$ and any vertex algebra  module. 
\end{lemma} 
\begin{lemma}\label{lm_VirQ}
For $T$ defined above depending on $Q$, the action of the $L_0$ and $L_{-1}$ elements of $Vir_c$ on a general element $u e^{{\phi}_\beta}$ of the VOA $\V_\Lambda$, is given by: 
\begin{align*}
&L_{-1} u e^{{\phi}_\beta} = \partial. (u e^{{\phi}_\beta}) \\
&L_0 u e^{{\phi}_\beta} = \left( \frac{(\beta, \beta)}{2} - (\beta, Q) + \left|u\right| \right)  u e^{{\phi}_\beta}
\end{align*}	
The $L_0$ eigenvalue is called \emph{conformal dimension} or energy.
\end{lemma}

\noindent
$\V_\Lambda=\V_{[0]}$ and any module $\V_{[\mu]},[\mu]\in \Lambda^*/\Lambda$ decompose as Virasoro module into it's $\Lambda$-grading layers
$$\V_{[\mu]}=\bigoplus_{\lambda\in[\mu]} \V_\lambda,\qquad \V_\lambda=\{u\exp{\phi_\lambda}\}$$
The operators $L_n$ act by $\N_0$ degree $-n$ and the unique vector of lowest degree in each $V_\lambda$ is clearly $\exp{\phi_\lambda}$. However, the Virasoro module $V_\lambda$  is neither irreducible nor generated by this lowest-degree vector (e.g. $L_{n}\exp{0}=0$ for all $n\neq 0$), rather it has a complicated structure of an indecomposable, that has been determined using screening operators. \\

\noindent
For each $\V_\Lambda$-module $\V_{[\mu]}$ the conformal dimension of any lowest-degree vectors $\exp{\phi_\lambda}$ is  
$$\frac{1}{2}(\lambda, \lambda) - (\lambda, Q) = \frac{1}{2}\Vert Q- \lambda \Vert^2 - \frac{1}{2}\Vert Q,Q\Vert^2$$
so these elements form a paraboloid with maximum in $\lambda = Q$. See below for the picture for the example $\sl_2$.

\subsection{Rescaled root lattices, short/long screenings}\label{sec_rescaledLattices}

Let $\g$ be a complex finite-dimensional semisimple Lie algebra. We denote by 
\begin{itemize} 
\item $\Lambda_R$ its root lattice with base the simple roots $\left\lbrace \alpha_1, \ldots , \alpha_{\rank}\right\rbrace $
\item $\Lambda_W$ its weight lattice with base the fundamental weights $\left\lbrace \lambda_1, \ldots , \lambda_{\rank}\right\rbrace $, 
where $(\lambda_i, \alpha_j) = d_j \delta_{ij} = \frac{1}{2}(\alpha_j, \alpha_j) \delta_{ij}$
\item ${\Lambda_R}^\vee$ its coroot lattice with base $\left\lbrace {\alpha_1}^\vee, \ldots , {\alpha_{\rank}}^\vee\right\rbrace $ where ${\alpha_i}^\vee := \alpha_i \frac{2}{(\alpha_i, \alpha_i)}$
\end{itemize}
Let $\ell\in \N$, such that $\ell$ is divisible by all $(\alpha_j, \alpha_j)$, in particular $\ell=2p$ is even.\\

We rescale the root lattice $\Lambda_R$:
\begin{definition} Define the \emph{short screening momenta} by
$$\left\lbrace {\alpha_1}^\oshort := -\alpha_1/\sqrt{p}, \; \ldots , \; {\alpha_{\rank}}^\oshort := -\alpha_{\rank}/\sqrt{p}\right\rbrace$$
This is a distinguished basis of the \emph{short screening lattice} 
$$\Lambda^\oshort := \frac{1}{\sqrt{p}}\Lambda_R$$
We choose this rescaling such that
$$q_{ij}:=e^{\pi\i({\alpha_i}^\oshort,{\alpha_j}^\oshort)}=q^{(\alpha_i,\alpha_j)}\qquad q:=e^{\frac{2\pi\i}{\ell}}$$
\end{definition}
We now consider the (generalized) lattice VOA  $\V_{\Lambda^\oshort}$. For every value of $Q$ there is a Virasoro action on the algebra defined by Lemma \ref{lm_Q}. The following choice is a slight generalization of \cite{FT} to non-simply-laced $\g$:
\begin{lemma}\label{valueQ}
The following unique $Q$ has the property that the conformal dimension of these basis elements $h^Q({\alpha_i}^\oshort)$ is equal to $1$
$$Q = (p\cdot{\rho_\g}^\vee - \rho_\g) / \sqrt{p}$$
where $\rho_\g$ is the half sum of all positive roots and ${\rho_\g}^\vee$ analogous for the dual root system. 
\end{lemma}
We remark that $Q$ should be seen as the center point in the fundamental alcove for the Lie algebra $\g$ in characteristic $\ell$, which is in our case not a prime.
\begin{proof}
By Lemma \ref{lm_VirQ} we have to solve the system of linear equations 
\begin{displaymath} \frac{1}{2} ({\alpha_i}^\oshort, {\alpha_i}^\oshort) - ({\alpha_i}^\oshort, Q) = 1
\end{displaymath}
We easily check that the asserted $Q$ solves this system: 
\begin{align*}
&\frac{1}{2} (-\frac{\alpha_i}{\sqrt{p}}, -\frac{\alpha_i}{\sqrt{p}}) - (-\frac{\alpha_i}{\sqrt{p}}, (p\cdot{\rho_\g}^\vee - \rho_\g) \frac{1}{\sqrt{p}}) \\
&= \frac{1}{2p} (\alpha_i, \alpha_i) + \frac{1}{p} (\alpha_i, p\cdot{\rho_\g}^\vee) - \frac{1}{p} (\alpha_i,\rho_\g)  \\
&= \frac{1}{2p} (\alpha_i, \alpha_i) + (\alpha_i, {\rho_\g}^\vee) - \frac{1}{p} \frac{(\alpha_i, \alpha_i)}{2} \\
&= (\alpha_i, {\rho_\g}^\vee)  = 1
\end{align*}
where we used that ${\alpha_i}^\vee = \alpha_i  \frac{2}{(\alpha_i, \alpha_i)}$ and $(\alpha_i,\rho_\g) = \frac{(\alpha_i, \alpha_i)}{2}$ 
\end{proof}
We remark that this definition depends on the choice of a set of positive simple roots for $\g$. A different Weyl chamber of $\g$ will give rise to a different vector $Q$ (the respective reflection), which gives rise to a different Virasoro action, however of the same central charge. 

\begin{lemma}\label{alphaoplus} Fixed $Q = (p\cdot{\rho_\g}^\vee - \rho_\g) / \sqrt{p}$ as above, then there is a second set of \emph{long screening momenta}
$$\left\lbrace {\alpha_1}^\olong := +\alpha_1^\vee\sqrt{p}, \; \ldots , \; \alpha_{\rank}^\olong := +\alpha_{\rank}^\vee\sqrt{p}\right\rbrace$$
They are a distinguished basis of the \emph{long screening lattice}:
$$\Lambda^\olong:=\sqrt{p}\Lambda_{R}^\vee$$
From the condition $(\alpha_i,\alpha_i)|\ell=2p$ it follows that $\Lambda^\olong$ is an integral lattice and the dual lattice (that determine the modules of $\V_{\Lambda^\olong}$) is
$$(\Lambda^\olong)^* = \frac{1}{\sqrt{p}} \Lambda_W$$
\end{lemma}
\begin{proof}
Again we compute $h^Q(\beta)$ by Lemma \ref{lm_VirQ}:
\begin{align*}
&\frac{1}{2} ({\alpha_i}^\vee \sqrt{p}, {\alpha_i}^\vee \sqrt{p}) - ({\alpha_i}^\vee \sqrt{p}, (p\cdot{\rho_\g}^\vee - \rho_\g)/\sqrt{p}) \\
&= \frac{p}{2} ({\alpha_i}^\vee, {\alpha_i}^\vee) - ({\alpha_i}^\vee, p\cdot{\rho_\g}^\vee) + ({\alpha_i}^\vee, {\rho_\g}) \\
&= \frac{p}{2} ({\alpha_i}^\vee, {\alpha_i}^\vee) - \frac{p}{2} ({\alpha_i}^\vee, {\alpha_i}^\vee) + \frac{2}{(\alpha_i, \alpha_i)} (\alpha_i, \rho_\g) = 1 
\end{align*}
\end{proof}
The inclusion relationships among these three lattices is as follow:
$$\Lambda^{\olong} \subset \Lambda^{\oshort} \subset {(\Lambda^{\olong})}^*$$
\begin{corollary}\label{formula_numberReps}
To compute the number of representations $|{(\Lambda^{\olong})}^*/\Lambda^{\olong}|$ in general it is helpful to use the intermediate lattice $\Lambda^{\oshort}$ as follows 
\begin{align*}
 |{(\Lambda^{\olong})}^*/\Lambda^{\olong}|
 &=|{(\Lambda^{\olong})}^*/\Lambda^{\oshort}|\cdot |{(\Lambda^{\oshort})}/\Lambda^{\olong}|\\
 &=|\frac{1}{\sqrt{p}} \Lambda_W / \frac{1}{\sqrt{p}}\Lambda_R |\cdot |\frac{1}{\sqrt{p}}\Lambda_R / \sqrt{p}\Lambda_R^\vee|\\
 &=|\Lambda_W / \Lambda_R|\cdot \prod_{i=1}^{\rank} \frac{\ell}{(\alpha_i,\alpha_i)}
\end{align*}
where $\Lambda_W / \Lambda_R$ is the fundamental group of the Lie algebra; the order is the determinant of the Cartan matrix. If $d=1,2,3$ is the half normsquare of a long root and,  then the number of representations is equivalently given by $(p/d)^\rank\cdot \det(\alpha_i,\alpha_j)$. 
\end{corollary}
Of course this calculation gives actually rise to a short exact sequence of abelian groups
$$0
\longrightarrow\prod_{i=1}^{\rank} \Z_{\frac{\ell}{(\alpha_i,\alpha_i)}}
\longrightarrow {(\Lambda^{\olong})}^*/\Lambda^{\olong} 
\longrightarrow \Lambda_W / \Lambda_R
\longrightarrow 0$$
In the application this gives a decomposition of $\V_{\Lambda^\olong}\md\Mod=\Vect_{{(\Lambda^{\olong})}^*/\Lambda^{\olong} }$ into blocks that are reachable by short screenings. The actual representation-theoretic blocks in our upcoming non-semisimple category will be a subdivision of this decomposition.

\subsection{Vertex algebra implications of the rescaled lattices}

Let $\g$ be a complex semisimple Lie algebra and $\ell\in\N$ divisible by all $(\alpha_i,\alpha_j)$. We apply the lattices and bases from the previous subsection to the vertex algebra setting:\\

Let $\V_{\Lambda^\olong}$ be the lattice vertex algebra associated to the even integer lattice. We know that the simple vertex algebra modules are $\V_{[\lambda]}$ parametrized by $(\Lambda^\olong)^*/\Lambda^\olong$ with 
$$\V_{(\Lambda^\olong)^*}=\bigoplus_{[\lambda]\in (\Lambda^\olong)^*/\Lambda^\olong} \V_{[\lambda]}
\qquad\qquad (\Lambda^\olong)^*/\Lambda^\olong=\frac{1}{\sqrt{p}}\Lambda_W/\sqrt{p}\Lambda_R^\vee$$

Theorem \ref{teo6.1} applied to the short screening momenta $\alpha^\oshort=\alpha_i/\sqrt{p}$ shows:
\begin{corollary}
 Assume $p\geq(\alpha_i,\alpha_i)$ (and this fails in our specific degenerate example for long roots!) then the \emph{short screening operators} $\zem_{{\alpha_i}^\oshort}$ constitute a representation on $(\V_\Lambda^\olong)^*$ of the Nichols algebra with braiding
 $$q_{ij}:=e^{\pi\i({\alpha_i}^\oshort,{\alpha_j}^\oshort)}=q^{(\alpha_i,\alpha_j)}\qquad q:=e^{\frac{2\pi\i}{\ell}}$$
 which is precisely the positive part of the small quantum group $u_q(\g)^+$. These are hence linear maps between different $\V_{\Lambda^\olong}$-representations 
 $$\zem_{{\alpha_i}^\oshort}:\;\V_{[\lambda]}\longrightarrow \V_{[\lambda+\alpha_i^\oshort]}$$
 Together with the following exponentiated grading operators $K_i$ they constitute a representation of $u_q(\g)^{\geq 0}$:
 $$K_{\alpha_i^\oshort}:=e^{+\pi\i\;\resY(\partial\phi_{\alpha_i/\sqrt{p}})}
 \qquad K_{\alpha_i^\oshort}\;u\exp{\phi_\lambda}=e^{\frac{2\pi\i}{\ell}(\alpha_i,\lambda\sqrt{p})}\cdot u\exp{\phi_\lambda}$$
 Note that these act as single scalars on each module $\V_{[\lambda]}$.
\end{corollary}
 On the other hand a well-established fact in \cite{FT} for simply-laced $\g$, and conjectured by the authors for the general case is 
\begin{corollary}
 The \emph{long screening operators} preserving each $\V_{\Lambda^\olong}$-representations 
 $$\zem_{{\alpha_i}^\olong}:\;\V_{[\lambda]}\longrightarrow \V_{[\lambda]}$$
 constitute a representation of the negative part of the dual Lie algebra $U(\g^\vee)^-$. Together with the following grading operators $H_i$ they constitute a representation of $U(\g^\vee)^{\geq 0}$.
 $$H_{\alpha_i^\olong\;d/p}:=\resY(\partial\phi_{-\alpha_i^\vee\sqrt{p}\;d/p})
 \qquad H_{\alpha_i^\olong\;d/p}\;u\exp{\phi_\lambda}=(\alpha_i^\vee,\lambda/\sqrt{p})d\cdot u\exp{\phi_\lambda}$$
 where $d\in\{1,2,3\}$ depending on the root system.
\end{corollary}

\noindent
We chose our Virasoro action parametrized by $Q$ such that all conformal dimensions are 
$$h^Q(\alpha_i^\oshort)=h^Q(\alpha_i^\olong)=1$$
Using the usual associativity of the OPE for the integer  operator $\zem_{\alpha^\olong}$, this implies 
\begin{corollary}
 The long screening operators $\zem_{\alpha_i^\olong}$ commute with the Virasoro action, and equivalently the energy-momentum tensor is in the kernel $\zem_{\alpha_i^\olong}(T^Q)=0$.
\end{corollary}
We would morally expect that the same is true for the short screenings, but this is wrong because of fractionality one easily sees that $\zem_{\alpha_i^{short}}$ usually does not preserve e.g. the $\N_0$-grading and hence $L_0$. However it is proven in \cite{LentQGNA} Lemma 6.4 using Nichols algebras that:
\begin{corollary}\label{cor_Weylaction}
  On a given module $\V_{[\lambda/\sqrt{p}]},\;[\lambda/\sqrt{p}]\in (\Lambda^\olong)^*/\Lambda^\olong$ and for a given short screening $\alpha_i^\oshort=-\alpha_i/\sqrt{p}$, we take  $0\leq k < \ell/(\alpha_i,\alpha_i)$ the unique natural number such that 
  $$-(\alpha_i,\lambda)+(k-1)(\alpha_i,\alpha_i)\in \ell\Z$$
  Then the $k$-th power of the short screening operator acting on this module 
  $$(\zem_{\alpha_i^\oshort})^k:\; \V_{[\lambda/\sqrt{p}]}\rightarrow \V_{[(\lambda-k\alpha_i)/\sqrt{p}]}$$
  is a finite sum. This gives the dot-action of the Weyl group.\\
  
  It is curently ongoing work to prove in general that these operators have an OPE formula resembling the integer case, which would in particular imply they commute with the Virasoro action and with $\Y(a)$ for $a\in \Ker(\zem_{\alpha_i^\oshort})^k$. This is a generalization and conceptualization of the main idea behind the Felder complex differential \cite{Fel89}. In the present article the only nontrivial cases action is on modules where the scalar product $(\alpha_i^\oshort,\lambda/\sqrt{p})$ is still an odd integer, so we can assume these assertions as classical facts.
\end{corollary}
For example, in the vacuum module $\V_{[0]}=\V_{\Lambda^\olong}$ we have $k=1$, so on this module all $\zem_{\alpha_i^\oshort}$ are well behaved. In particular they commute with the Virasoro action and the energy-momentum tensor is in the kernel $\zem_{\alpha_i^\oshort}(T^Q)=0$.


\section{The conjectures on the screening charge method for LCFTs} 
We now want to use the short screenings to construct a subspace $\mathcal{W}$ of the lattice VOA $\V_{\Lambda^{\olong}}$. The idea is that this realizes an interesting VOA with a non-semisimple representation theory, giving a free-field realization of a logarithmic conformal field theory. The construction of such a $\mathcal{W}$ is interesting in its own right, since few logarithmic CFTs are known: Most importantly the triplet algebra $\W_p$, which has been shown to arise from the present construction for the special case $A_1$, and the even part of $n$ pairs of symplectic fermions, which we show in this article to correspond to $B_n,\ell=2p=4$.
\begin{definition}
The subject of study is the subspace of the lattice VOA $\V_{\Lambda^\olong}$
\[\mathcal{W} :=  \bigcap\limits_i \ker_{\V_{\Lambda^\olong}} \zem_{{\alpha_i}^\oshort}.\]
\end{definition}

The following famous conjecture are in place, formulated successively by several authors such as \cite{DF84,FF,FFHST02,FGSTsl2,FT, AM08} and somewhat extended by the second author
\begin{conjecture}\label{conj_Main}~
\begin{enumerate}
	\item $\mathcal{W}$ is a vertex subalgebra of $\V_{\Lambda^\olong}$. Similarly the kernels of suitable powers $(\zem_{\alpha_i^\oshort})^k$ on each module $\V_{[\lambda]}$ should be modules of $\mathcal{W}$. 
	\item The action of $U(\g^\vee)^-$  via long screenings on $\W$ extends to an action of $\g^\vee$, and short and long screenings together give a representation of Lusztig's infinite-dimensional quantum group of divided powers $U_q^\L(\g)$, see section \ref{sec_degenerateLusztig}.  
	\item $\W$ is a logarithmic VOA  i.e. a VOA with finite non-semisimple representation theory. This would follow from the next conjecture.
	\item The representation category of $\W$ as abelian category is equivalent  to representations of the respective small quantum group $u_q(\g)$ with $q=e^{\frac{2\pi\i}{\ell}}$. Similarly, the kernel of both long and short screenings (and related, the $0$-degree $\W\cap\V_0$), should be the quantum Drinfel'd Sokolov reduction of the affine Lie algebra $\hat{\g}$, with the representation theory related to $U_q^\L(\g)$.
	\item The representation category of $\W$ as braided tensor category should be equivalent to the representations of a quasi-Hopf algebra $u_q^\omega(\g)$ with Cartan part $(\Lambda^\olong)^*/\Lambda^\olong$ as discovered in \cite{GR} for type $A_1$ and in general as described in our paper \cite{GLO17}.
\end{enumerate} 
\end{conjecture}
\begin{remark}
 We remark on each of these conjectures:
\begin{enumerate}
	\item Both of these assertions should be proven by using Corollary \ref{cor_Weylaction} these powers are well-behaved and have an OPE associativity formula similar to the case of usual (integer-power) vertex operators.  
	\item Presumably the action of the other half of $U_q(\g)$ can be constructed by regularizing the vanishing short screening powers $$\frac{(\zem_{-\alpha_i/\sqrt{p}})^{\ell/(\alpha_i,\alpha_i)}}{[\ell/(\alpha_i,\alpha_i)]_q!}$$
	Here again one has to prove that regularized short screening powers have an OPE associativity formula similar to the case of usual (integer-power) vertex operators, then the relations of $\g$ follow quite straight-forward.  
	\item This would follow from the next conjecture.
	\item The category equivalence is by all means the hardest and most interesting conjecture. One possible proof strategy is to construct a projective bimodule with commuting action of $\W$ and $u_q(\g)$ as in \cite{FGSTsl2} for $A_1$. In the general case one might want to obtain the bimodule by adding all Weyl reflections of $\V_{[0]}$ and regularizing them on the kernel of the screening to obtain logarithmic module extensions. Alternatively one might directly construct a categorical Weyl group action on the representation category, which essentially means to retake the steps in \cite{AJS94}. Hence the resulting equivalence can be understood as a CFT variant of Kazhdan-Lusztig correspondence.   
\end{enumerate} 
\end{remark}

\section{Known example $A_1,\ell = 4$}

In the next sections we determine the decomposition behaviour of the simple lattice VOA modules $\V_{[\lambda]}$, when restricted to the kernel of the short screenings $\mathcal{W}\subset\V_{\Lambda^\olong}$. We first do the known example $A_1,\ell=4$, which is the triplet algebra $\W_2$ and at the same time the even part of a single pair of symplectic fermions. Since $B_1=A_1$ this is also the smallest case of our family of examples $B_n,\ell=4$ which we treat subsequently.

\subsection{Lattices}

Let us apply what we have described to the easiest example $\g = \sl_2$ with root system $A_1$, which is treated in \cite{FFHST02,FGSTsl2,TW,TN}:

The root lattice is $\Lambda_R = \alpha \mathbb{Z}$ with $(\alpha, \alpha) = 2$. Let $\ell = 2p$ be an arbitrary even positive number. We compute all the data of the last two sections:

\begin{itemize}
	\item The \emph{short screening momentum} spanning the short screening lattice is  
	$$\alpha^\oshort=-\frac{\alpha}{\sqrt{p}}
	\qquad\quad {\Lambda}^\oshort = \frac{1}{\sqrt{p}}\Lambda_R = \frac{\alpha}{\sqrt{p}}\mathbb{Z}=\frac{2}{p}\Z$$
	\item  To compute $Q$ as in Lemma \ref{valueQ} we notice that $\alpha^\vee = \alpha$ and 
		\begin{displaymath}
			\rho_\g = {\rho_\g}^\vee = \frac{1}{2} \alpha
		\end{displaymath}
		so we easily arrive to
		\begin{displaymath}  
			Q = \frac{1}{\sqrt{p}}(p {\rho_\g}^\vee - \rho_\g) = \frac{1}{2\sqrt{p}} (p\cdot \alpha - \alpha) = \frac{1}{2}\frac{(p - 1)}{\sqrt{p}} \alpha = - \frac{p-1}{2} {\alpha}^\oshort
		\end{displaymath}

	\item The \emph{long screening momentum} spanning the long screening lattice is  
	$$\alpha^\olong=+\alpha \sqrt{p}
	\qquad\quad {\Lambda}^\olong = \sqrt{p}\Lambda_R^\vee = \alpha\sqrt{p}\mathbb{Z}= 2p\Z$$
	The dual lattice determining the $\V_{\Lambda^\olong}$-modules is spanned by the rescaled fundamental weight
	$$\lambda/\sqrt{p}=\alpha/2\sqrt{p}  
	\qquad\quad (\Lambda^\olong)^*= \frac{1}{\sqrt{p}} \Lambda_W = \frac{\alpha}{2\sqrt{p}} \mathbb{Z}=\frac{1}{2p}\Z$$
	Now we can compute the number of representations of $\V_{\Lambda^\olong}$ to be 
		\begin{align*} \left| (\Lambda^\olong)^* / \Lambda^\olong\right| &= 
		\left| (\Lambda^\olong)^* / \Lambda^\oshort\right| \cdot \left| \Lambda^\oshort / \Lambda^\olong\right|  = \\
		&= \left|\frac{1}{\sqrt{p}} \Lambda_W  / \frac{1}{\sqrt{p}} \Lambda_R \right| \cdot \left|\frac{1}{\sqrt{p}} \Lambda_R / \sqrt{p} {\Lambda_R}^\vee\right|  =\\
		&=  2\cdot p 			
		\end{align*}

	\item Explicitly these $2p$ representations are given by the $\Lambda^\olong$-cosets with representatives: 
		\begin{displaymath}
			(\Lambda^\olong)^* / \Lambda^\olong = \left\{ [k \alpha/2\sqrt{p}] : k = 0,1,\ldots, 2p-1\right\}
		\end{displaymath}
	that fall into two $\Lambda^\oshort$-cosets
	$$0 + \Lambda^\oshort / \Lambda^\olong = \{ k \alpha/2\sqrt{p}, \mbox{k even}\} 
	\qquad \frac{\alpha}{2\sqrt{p}} + \Lambda^\oshort / \Lambda^\olong = \{ k \alpha / 2 \sqrt{p}, \mbox{k odd}\}$$
	The representations of $\V_{\Lambda^\olong}=\V_{\sqrt{2p}\Z}$ are then $\V_{[\frac{k \alpha}{2 \sqrt{p}}]}$, with $k = 0,1,\ldots,2p-1$.
\end{itemize}

From now on we always set $\ell=2p=4$. We draw a picture of all elements in all simple modules  
	$$\V_{(\Lambda^{\olong})^*}=\{u\exp{\phi_{k \alpha / 2 \sqrt{p}}}\}=\bigoplus_{k\in\Z}\V_{k\alpha / 2 \sqrt{p}}$$
	with the $\Lambda$-grading by $k$ on the X-axis and the conformal dimension on the Y-axis. At the top of each pyramid is a pure exponential $\exp{\phi_{\lambda}},\;\lambda=k \alpha / 2 \sqrt{p}$ as a larger dot, one level below we the one basis element $\partial\phi \exp{\phi_{\lambda}}$, another level below the two basis elements 
	$\partial\phi\partial\phi \exp{\phi_{\lambda}},\partial^2\phi \exp{\phi_{\lambda}}$, then three basis elements, then five basis elements etc.\\
	
	By $k\;\mod\;4$ the elements are grouped into the $2\cdot 2$ many $\V_{\Lambda^\olong}$-modules 
	$$\V_{(\Lambda\olong)^*}=\V_{[0]}\oplus\V_{[\alpha / 2 \sqrt{p}]}\oplus\V_{[2\alpha / 2 \sqrt{p}]}\oplus\V_{[3\alpha / 2 \sqrt{p}]}$$
	and we draw the vacuum module $\V_{[0]}=\V_{\Lambda^\olong}$ with solid boundary, the module $\V_{[2\alpha / 2 \sqrt{p}]}=\V_{[\alpha / \sqrt{p}]}$ without boundary; the two modules in the other  $\Lambda^\oshort$-coset are not shown, only the first two top exponential $k=-1,+1,+3$ in grey.
	\begin{center}
	 \includegraphics[scale=.3]{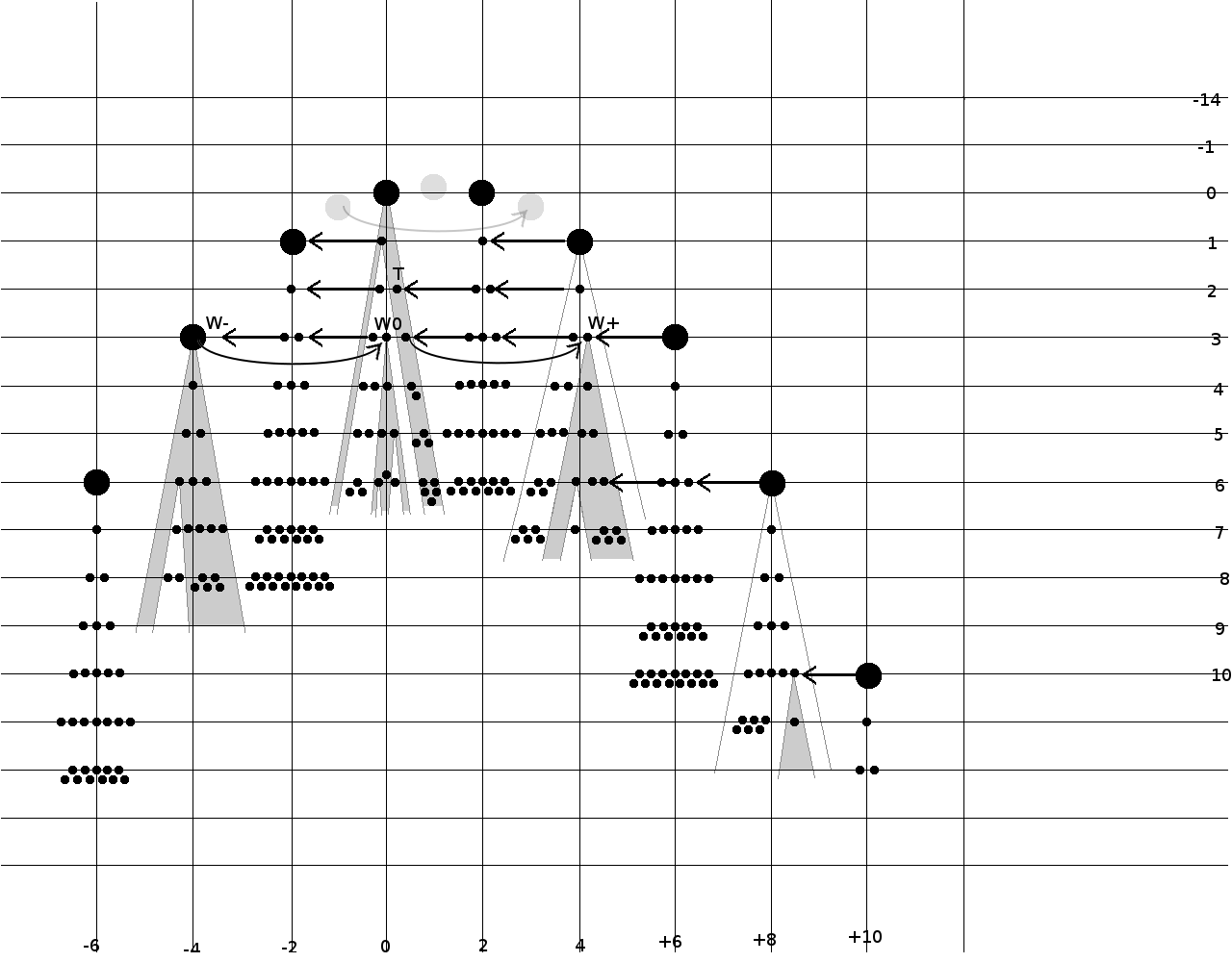}
	\end{center}
	We have also drawn some short screenings (right-to-left) and long screenings (left-to-right) and for the purpose of the next section we have drawn in shaded grey the kernel of the short screening in the vacuum module  $\V_{[0]}=\V_{\Lambda^\olong}$, including the energy-stress-tensor $T$.


\subsection{Groundstates}

\begin{definition}
We define the \emph{groundstates} of a vertex algebra module as the vector subspace of lowest conformal dimension. In particular for a lattice VOA $\V_\Lambda$ by Lemma \ref{lm_VirQ} the groundstates of a module $\V_{[\lambda]}$ is spanned by the elements $e^{\phi_\lambda}$ elements for those representatives $\lambda$ with minimal distance to the vector $Q$ defining the Virasoro action. For $A_1,\ell=4$ we give informal names to them:
\emph{Blue} (or \emph{Vacuum}), \emph{Center}, \emph{Green} and \emph{Steinberg} module.

\end{definition}
In the following table we will write in the case $A_1,\ell=4$ for each module the dimension of the groundstates, the respective conformal dimension and explicitly the elements of the groundstates. 
\[
\begin{array}{ll|lll}
  && \#\mbox{Groundstates} & \mbox{Conformal dimension} & \mbox{Groundstates } e^{\phi_\lambda} \\
\hline
\text{Blue} & \V_{[0]}  & 1 & 0 &  0\\
\text{Center} &\V_{[1]} & 1 & -1/8 & \alpha/2\sqrt{2}=Q\\
\text{Green} &\V_{[2]} & 1 & 0 & \alpha/\sqrt{2}\\
\text{Steinberg} &\V_{[3]} & 2 & 3/8 & -\alpha/2\sqrt{2},\; 3\alpha/2\sqrt{2}
\end{array}
\]

We now focus on the screening operators and on their action. From Corollary \ref{cor_Weylaction} we know that the \emph{long screening operators} $\zem_{\alpha_i^\olong}$ and some power of the \emph{short screening operators} $(\zem_{\alpha_i^\oshort})^k$ are \emph{Vir}-homomorphisms.
In particular on the Blue and Green modules this power is $k=1$, while on the Center module the power is $k=0$ (so just the identity) and on the Steinberg module the power is $k=2$ (so identically zero by the Nichols algebra relation in Corollary \ref{cor_NicholsRelations}), i.e. on them the \emph{short screening} operators are really \emph{Vir}-homomorphisms.

\begin{remark}
For $A_1$ with general $p$ the dimension of the groundstates is again two in the facet module $\V_{[-1]} = \V_{[2p-1]}$ module and is one in all the other $2p-1$ modules. The Weyl group action  $(\zem_{\alpha_i^\oshort})^k$ around $Q$ connects $\V_{p-1-k}\leftrightarrow \V_{p-1+k}$, it is the identity on the Center module $\V_{p-1}$ and identical zero on the Steinberg module. 

The groundstates are a fundamental domain (or alcove) for $\Lambda^\olong$-coset representatives and $\V_{[-1]} = \V_{[2p-1]}$ is on the boundary (facet). This picture continues to hold for arbitrary $\g$, where we have regular weights with a single groundstate i.e. representative inside the fundamental alcove, and singular weights at the facets.\\

\end{remark}

\subsection{Screening operators}

In this setting and the explicit expression for the screening operator \ref{def_screening} we can explicitly apply the short and long screening operators $\zem_{-\alpha/\sqrt{2}}$ and $\zem_{\alpha \sqrt{2}}$ to our modules to see how they act. Since in this specific example on the modules $[0], [\frac{2\alpha}{2\sqrt{2}}]$ integrality still holds $(-\frac{\alpha}{\sqrt{2}}, \lambda) \in \mathbb{Z}$. We start with some examples for the short screening operator:
\begin{enumerate}
	\item The short screening operator applied to the \emph{Blue} groundstates element $e^0=1$ (the vacuum) with conformal dimension $0$ give: \\
				$\zem_{-\alpha/\sqrt{2}}(e^0)  = 0 $
	\item The short screening operator applied to the second layer with conformal dimension $1$ of the \emph{Blue} module give:
	\begin{align*}
	 \zem_{-\alpha/\sqrt{2}}(\partial \phi_{\alpha/\sqrt{2}}) 
	 &=   \sum\limits_{k} \left\langle e^{{\phi}_{-\frac{\alpha}{\sqrt{2}}}}, \partial \phi_{\alpha/\sqrt{2}} \right\rangle_{-k-1} \frac{1}{k!} \partial^k e^{{\phi}_{-\frac{\alpha}{\sqrt{2}}}}  \\
	 &+\sum\limits_{k} \left\langle e^{{\phi}_{-\frac{\alpha}{\sqrt{2}}}}, 1 \right\rangle_{-k-1} \partial \phi_{\alpha/\sqrt{2}} \frac{1}{k!} \partial^k e^{{\phi}_{-\frac{\alpha}{\sqrt{2}}}}\\
	 &=\left\langle e^{{\phi}_{-\frac{\alpha}{\sqrt{2}}}}, \phi_{\alpha/\sqrt{2}} \right\rangle_{-1} \partial  e^{{\phi}_{-\frac{\alpha}{\sqrt{2}}}}+0\\
	 &=-(\frac{\alpha}{\sqrt{2}}, -\frac{\alpha}{\sqrt{2}}) e^{{\phi}_{-\frac{\alpha}{\sqrt{2}}}} = e^{{\phi}_{-\frac{\alpha}{\sqrt{2}}}}\\ 
	\zem_{-\alpha/\sqrt{2}}(e^{\phi_{\alpha\sqrt{2}}}) 
	&= \sum\limits_{k} \left\langle e^{{\phi}_{-\frac{\alpha}{\sqrt{2}}}}, 1 \right\rangle_{-k-1 +2} e^{\phi_{\alpha\sqrt{2}}} \frac{1}{k!} \partial^k e^{{\phi}_{-\frac{\alpha}{\sqrt{2}}}} \\
	&= \partial \phi_{-\alpha/\sqrt{2}} e^{{\phi}_{\frac{\alpha}{\sqrt{2}}}}\\
	\end{align*}

	\item The short screening operator applied to the \emph{Green} groundstates element with conformal dimension $0$ gives: 
	\begin{align*}
	 \zem_{-\alpha/\sqrt{2}}(e^{{\phi}_{\frac{\alpha}{\sqrt{2}}}}) 
	 &=  \sum\limits_{k} \left\langle e^{{\phi}_{-\frac{\alpha}{\sqrt{2}}}}, 1 \right\rangle_{-k-1+1} e^{{\phi}_{\frac{\alpha}{\sqrt{2}}}} \frac{1}{k!} \partial^k e^{{\phi}_{-\frac{\alpha}{\sqrt{2}}}} \\
	 &=\left\langle e^{{\phi}_{-\frac{\alpha}{\sqrt{2}}}}, 1 \right\rangle_{0} e^{{\phi}_{\frac{\alpha}{\sqrt{2}}}}  e^{{\phi}_{-\frac{\alpha}{\sqrt{2}}}} 
	 \;=1
	\end{align*}
	\item The short screening operator applied to the second layer with conformal dimension $1$ of the \emph{Green} module give
	\begin{align*}
	 \zem_{-\alpha/\sqrt{2}}(\partial \phi_{\alpha/\sqrt{2}} e^{\phi_{\frac{\alpha}{\sqrt{2}}}}) 
	 &= \zem_{-\alpha/\sqrt{2}}(L_{-1}(e^{{\phi}_{\frac{\alpha}{\sqrt{2}}}}))\\ 
	 &= L_{-1}( \zem_{-\alpha/\sqrt{2}}(e^{{\phi}_{\frac{\alpha}{\sqrt{2}}}})) = L_{-1} (1) \;= 0\\
	 \zem_{-\alpha/\sqrt{2}}(e^{{\phi}_{-\frac{\alpha}{\sqrt{2}}}}) 
	 &=  \sum\limits_{k} \left\langle e^{{\phi}_{-\frac{\alpha}{\sqrt{2}}}}, 1 \right\rangle_{-k-1-1} e^{{\phi}_{-\frac{\alpha}{\sqrt{2}}}} \frac{1}{k!} \partial^k e^{{\phi}_{-\frac{\alpha}{\sqrt{2}}}} \; =0
	\end{align*}
	where we abbreviated the first calculation using the commutativity with the Virasoro action.
\end{enumerate}

If we apply the long screening operators to these elements and find in this example the conjectured action of $\sl_2$ on the kernel of short screenings
$$ \zem_{\alpha\sqrt{2}}(e^0) 
  = 0$$
$$
 \zem_{\alpha\sqrt{2}}(e^{{\phi}_{-\frac{\alpha}{\sqrt{2}}}}) 
 = \partial \phi_{\alpha \sqrt{2}}e^{{\phi}_{\frac{\alpha}{\sqrt{2}}}}
 \qquad\qquad \qquad\qquad  
 \zem_{\alpha\sqrt{2}}(\partial \phi_{\alpha/\sqrt{2}} e^{\phi_{\frac{\alpha}{\sqrt{2}}}}) 
 = 0$$
 $$ \zem_{\alpha\sqrt{2}}(T)=0$$
and in degree $3$ the adjoint representation $W^-,W^0,W^+$. \\

In contrast, applying a single short screening 
to the groundstate of the Center module shows the effect of fractionality if the prerequisite of Corollary \ref{cor_Weylaction} does not hold:
\begin{align*}
 \zem_{-\alpha/\sqrt{2}}(e^{\phi_{\alpha/2\sqrt{2}}})
 &=\res{\Y(\exp{\phi_{-\alpha/\sqrt{2}}})\exp{\phi_{\alpha/2\sqrt{2}}}}\\
 &=\sum_{k\geq 0}\sum_m \res{z^{k+m}}\cdot \langle\exp{\phi_{-\alpha/\sqrt{2}}},\exp{\phi_{\alpha/2\sqrt{2}}}\rangle_m 
 \cdot \exp{\phi_{\alpha/2\sqrt{2}}}\frac{\partial^k}{k!}\exp{\phi_{-\alpha/\sqrt{2}}}\\
 &= \sum_{k\geq 0} 
     \frac{e^{2\pi\i(k-\frac{1}{2}+1)}-1}{2\pi\i(k-\frac{1}{2}+1)}
     \cdot \exp{\phi_{\alpha/2\sqrt{2}}}\frac{\partial^k}{k!}\exp{\phi_{-\alpha/\sqrt{2}}}
\end{align*}
This is an infinite sum that does not even preserve the grading $L_0$.

\subsection{Restricting lattice modules to $\mathcal{W}$}
We now look at the subspace $\mathcal{W}$ and at its representations. In the case $A_1$ it is known to be the $\mathcal{W}_{p}$-triplet algebra generated by the elements shown in the picture, which form the three-dimensional adjoint representation under the action of $\g^\vee=\sl_2$ via long screenings: 
$$\mathcal{W}^-:=\exp{\phi_{-\alpha}},
\qquad \mathcal{W}^0=\zem_{\alpha\sqrt{p}}\exp{\phi_{-\alpha}},
\qquad \mathcal{W}^+=(\zem_{\alpha\sqrt{p}})^2\exp{\phi_{-\alpha}}$$ 

For $A_1$ there is only a single short screening $\zem_{-\alpha/\sqrt{2}}$, so we have only one short screening operator  and so the definition become: 
\begin{align*}
\mathcal{W} &=  \ker_{\V_{\Lambda^\olong}} \zem_{-\alpha/\sqrt{2}}
\end{align*}
The first few layers of $\W$ are (compare the calculations and the picture) 
$$\langle\exp{0}\rangle_\C
\;\oplus\; 0
\;\oplus\; \langle T^Q \rangle_\C 
\;\oplus\; \langle W^-, W^0,\partial T^Q, W^+\rangle_\C 
\;\oplus\;\cdots$$

We now restrict the modules of $\V_{\Lambda^\olong}$ to $\mathcal{W}$-modules and describe their (non-semisimple) decomposition behaviour into irreducibles: By Corollary \ref{cor_Weylaction} we have Virasoro homomorphism (Weyl group action). 
$$ {(\zem_{-\alpha/\sqrt{2}})}^k: \V_{[p-1 + k]} \longrightarrow \V_{[p-1 - k]}$$

These Virasoro homomorphisms commute by OPE associativity with $Y(a)$ with $a\in\Ker_{\V^\olong}(\zem_{-\alpha/\sqrt{2}})$, so they are vertex module homomorphisms for the restriction to $\W$. Thus their kernels and images are $\mathcal{W}$-submodules:

\begin{itemize}
	\item For the Center module $\V_{[1]}$ we have $k=0$ and the Virasoro homomorphism is thus just the identity, the kernel is trivial and the image is all. This module stay irreducible over $\mathcal{W}$-action and we call it 
	$$\V_{[1]}\cong \Lambda(2)$$ 
	\item For the Steinberg module $\V_{[3]}=\V_{[-1]}$  we have $k=2$ and  ${(\zem_{-\frac{\alpha}{\sqrt{2}}})}^2 = 0$, so the image is trivial and the kernel is all. Again the module stays irreducible and we call it 
	$$\V_{[3]}=\Pi(2)$$
	\item For the Blue and Green modules, respectively $\V_{[0]}$ and $\V_{[2]}$, we have $k=1$ and thus a single short screening $(\zem_{-\frac{\alpha}{\sqrt{2}}})$. The modules  decompose them into non-trivial kernel and image. We will call them $\Lambda(1)$ and $\Pi(1)$ and explicitly we have: 
	\begin{align*}
\W=\Lambda(1) \rightarrow &\V_{[0]} \rightarrow \Pi(1)\\
\Pi(1) \rightarrow &\V_{[2]} \rightarrow \Lambda(1)
\end{align*}
\end{itemize}

We remark that the letters were chosen by Semikhatov to visualize that $\Lambda(1),\Lambda(2)$ have a 1-dimensional groundstates $\exp{0}$ resp. $\exp{\phi_{\alpha/2\sqrt{2}}}$ and $\Pi(1), \Pi(2)$ have 2-dimensional groundstates $e^{\phi_{-\alpha/\sqrt{2}}},\partial e^{\phi_{+\alpha/\sqrt{2}}}$ resp.  $\exp{\phi_{-\alpha/2\sqrt{2}}},\exp{\phi_{3\alpha/2\sqrt{2}}}$. 

\begin{remark} We remark that this argument neither proves that $\Lambda(i), \Pi(i)$ are irreducible representations nor that these are all the irreducible representations of $\mathcal{W}$. However in the $A_1$ case it is proven, e.g. in \cite{AM08}.
\end{remark}

\section{Example $B_2,\ell = 4$}\label{B_2} 

The case $A_1$, in particular $\ell=2p=4$, is well known with the kernel of short screenings $\W=\W_p$ the triplet algebra. In the next section we discuss the concepts introduced above thoroughly in this the new example $B_n,\;\ell = 4$. We will ultimately prove that $\W$ is isomorphic as a vertex algebra to the even part of the vertex algebra of $n$ pairs of symplectic fermions. To understand the general instance we will start with the $B_2$ case.

\subsection{Lattices}
Let $\g=\so_5$ be a Lie algebra with root system $B_2$, with a set of positive simple roots $\{\alpha_1, \alpha_2\}$ and Killing form: 
$$(\alpha_i,\alpha_j)=\begin{pmatrix}
   4 & -2 \\
   -2 & 2
  \end{pmatrix}$$
so $\alpha_1$ is the long root and $\alpha_2$ the short root.
The coroots are ${\alpha_1}^\vee = \alpha_1/2,\;{\alpha_2}^\vee= \alpha_2$ and span a root system of type $C_2$. The sets of positive roots and positive coroots are 
\begin{align*}
&\Phi^+(B_2) = \{\alpha_1,\; \alpha_2,\; \alpha_1 + \alpha_2,\; \alpha_1 + 2\alpha_2 \} \\ 
&\Phi^+({B_2}^\vee) = \{{\alpha_1}^\vee,\; {\alpha_2}^\vee, \;{\alpha_1}^\vee + {\alpha_2}^\vee, \; 2{\alpha_1}^\vee + {\alpha_2}^\vee \}
\end{align*} 
Therefore
\begin{align*}
&\rho_\g = 3/2 \; \alpha_1 + 2 \alpha_2 \\
&{\rho_\g}^\vee = 2 {\alpha_1}^\vee + 3/2 \; {\alpha_2}^\vee = \alpha_1 + 3/2 \; \alpha_2
\end{align*}
The value of $Q$ and the central charge of the respective Virasoro action are 
\begin{align*}
 Q 
 &= \frac{1}{\sqrt{p}} (p {\rho_\g}^\vee - \rho_\g)  = \frac{1}{\sqrt{2}}(2 \alpha_1 + 2 \frac{3}{2} \alpha_2 - \frac{3}{2} \alpha_1 - 2 \alpha_2) = \frac{\alpha_1}{2\sqrt{2}} + \frac{\alpha_2}{\sqrt{2}} \\
 c
 &=\rank-12(Q,Q)=2-6=-4
\end{align*}
 The short screening lattice with its base is in this special case an odd integral lattice
	\begin{displaymath} 
	\Lambda^\oshort = \frac{1}{\sqrt{p}} \Lambda_R , 
	\qquad \{\alpha_1^{\oshort},\alpha_2^\oshort \}=\{-\frac{1}{\sqrt{2}}\alpha_1, -\frac{1}{\sqrt{2}} \alpha_2 \}
	\end{displaymath} 
The long screening lattice with its base is the even integral lattice 
	\begin{displaymath} 
	\Lambda^\olong = \sqrt{p} {\Lambda_R}^\vee , \qquad 
	\{\alpha_1^{\olong},\alpha_2^\olong \}=\left\{\sqrt{2} {\alpha_1}^\vee, \sqrt{2} {\alpha_2}^\vee\right\} 
	\end{displaymath} 
The dual of the  long screening lattice $(\Lambda^\olong)^*$ with its base is given by fundamental weights
	\begin{displaymath} 
	\hspace{2cm}
	(\Lambda^\olong)^* = \frac{1}{\sqrt{2}} \Lambda_W \qquad \{{\lambda_1}/\sqrt{2},
	 {\lambda_2}/\sqrt{2}\} = \{\frac{\alpha_1}{\sqrt{2}} + \frac{\alpha_2}{\sqrt{2}}, \qquad
	 \frac{\alpha_1}{2\sqrt{2}} + \frac{\alpha_2}{\sqrt{2}}\}
	\end{displaymath} 
	${\lambda_2}/\sqrt{2} = Q$ and will be the groundstates representative of one of the modules $\V_{[{\lambda_2}/\sqrt{2}]}$. 
We compute the number of representations of $\V_\Lambda^\olong$ using formula (\ref{formula_numberReps}):
\begin{align*}
 \left| (\Lambda^\olong)^* / \Lambda^\olong\right|
 &=|{(\Lambda^{\olong})}^*/\Lambda^{\oshort}|\cdot |{(\Lambda^{\oshort})}/\Lambda^{\olong}|\\
 &=|\frac{1}{\sqrt{2}} \Lambda_W / \frac{1}{\sqrt{2}}\Lambda_R |\cdot |\frac{1}{\sqrt{2}}\Lambda_R / \sqrt{2}\Lambda_R^\vee|\\
 &=|\Lambda_W / \Lambda_R|\cdot \prod_{i=1}^{\rank} \frac{\ell}{(\alpha_i,\alpha_i)}
 =2\cdot 2
\end{align*}
The $\V_{\Lambda^\olong}$-modules $\V_{[\lambda]}$ are given by the cosets in the quotient $(\Lambda^\olong)^* / \Lambda^\olong$:
\begin{displaymath}
[0] \qquad [{\lambda_2}/\sqrt{2}] \qquad [{\lambda_1}/\sqrt{2}] \qquad [{\lambda_1}/\sqrt{2} + {\lambda_2}/\sqrt{2}]
\end{displaymath}
As in $A_1=B_1$ we call these modules \emph{Blue} $\V_{[0]}$, \emph{Center} $\V_{[{\lambda_2}/\sqrt{2}]}$, \emph{Green} $\V_{[{\lambda_1}/\sqrt{2}]}$ and \emph{Steinberg} $\V_{[{\lambda_1}/\sqrt{2} + {\lambda_2}/\sqrt{2}]}$ module respectively. Again on the level of lattices $\V_{[0]},\V_{[2]}$ resp. $\V_{[1]},\V_{[3]}$ are connected by adding elements in the short screening lattice.

\begin{narrow}{-2cm}{0cm}
\begin{center}
\begin{tabular}{cc}
 \includegraphics{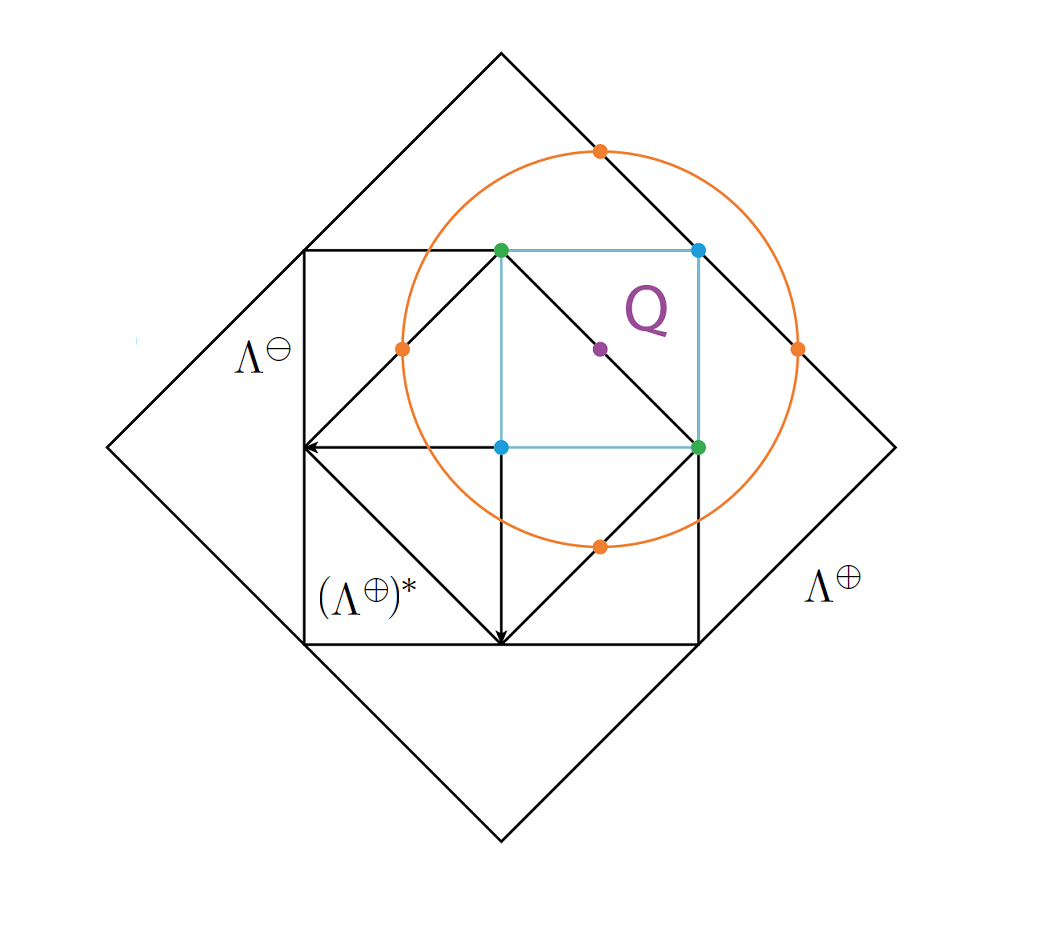} & \hspace{-1.8cm}
\includegraphics[scale=1]{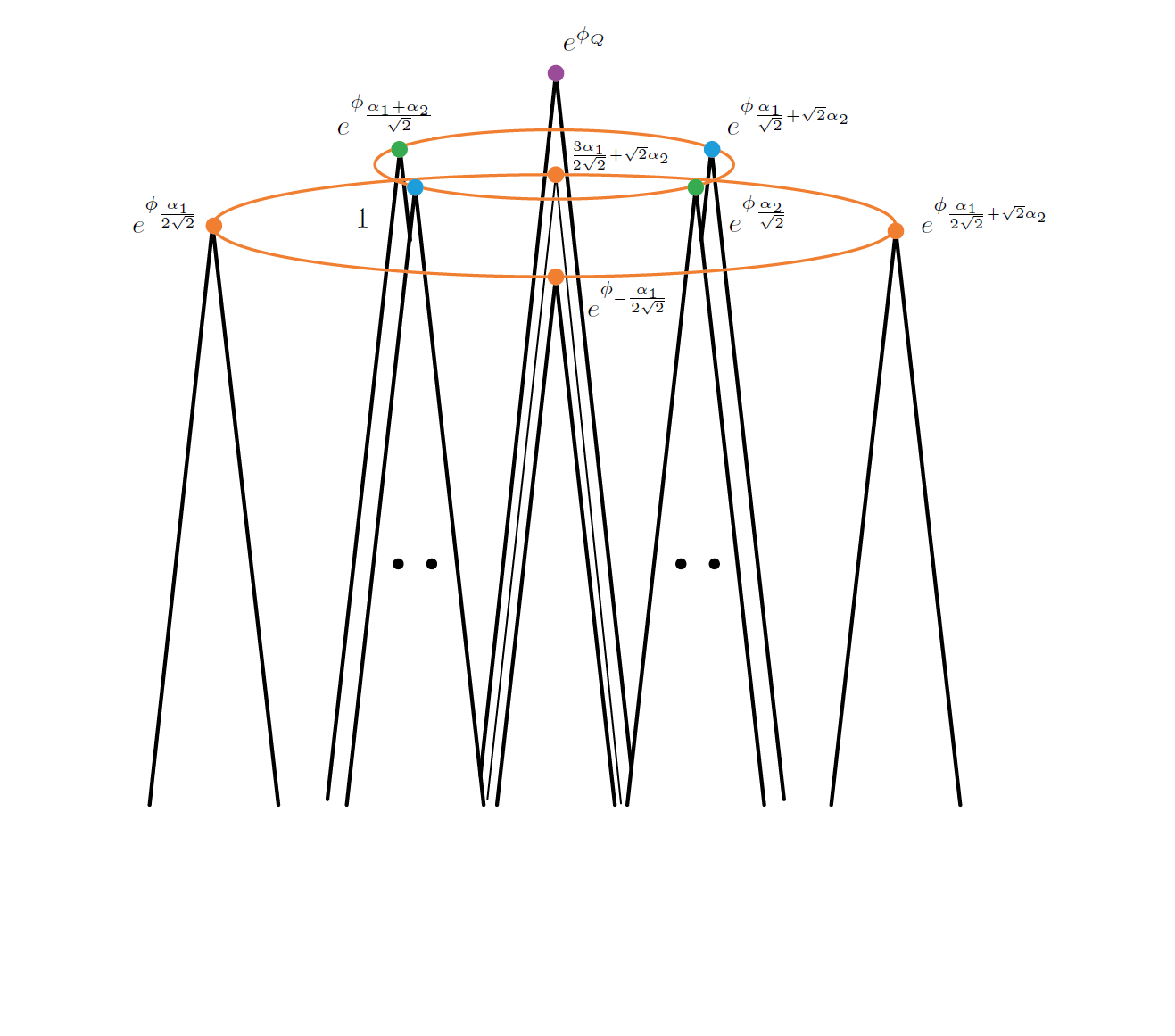}
\end{tabular}
\end{center}
\end{narrow}

\subsection{Groundstates}
Our next aim is to find the groundstates elements $e^{\phi_\alpha}$ for each module $\V_{[\lambda]}$ i.e.  the closest representatives of each $\Lambda^\olong$-coset $[\lambda]$ to $Q$. This is theoretically a hard computation problem, but easily solvable for root lattice. In the following table we write next to each module the number of groundstates, their conformal dimension and the groundstates elements:


\[
\begin{array}{l|llll}
 & \#\mbox{Groundstates} & \mbox{Conformal Dim} & \mbox{Groundstates elements } e^{\phi_\lambda}\\
\hline
\V_{[0]} & 2 & 0 & 0 \left(e^0 = \emph{1}\right), \; \alpha_1/ \sqrt{2} + \sqrt{2}\alpha_2 \\ 
\V_{[{\lambda_2}/\sqrt{2}]} & 1 & -1/4 & Q = \alpha_1/2\sqrt{2} + \alpha_2/ \sqrt{2} \\
\V_{[{\lambda_1}/\sqrt{2}]} & 2 & 0 & \alpha_1/\sqrt{2} + \alpha_2/ \sqrt{2}, \; \alpha_2/\sqrt{2}\\
\V_{[({\lambda_1}+\lambda_2)/\sqrt{2} ]} & 4 & 1/4 & 3\alpha_1/2 \sqrt{2} + \sqrt{2}\alpha_2, \; \alpha_1/2 \sqrt{2}, \\
 & & & -\alpha_1/2 \sqrt{2}, \; \alpha_1/2\sqrt{2} + \sqrt{2} \alpha_2 
\end{array}
\]


\subsection{Degeneracy}
As already mentioned $B_n$ is somewhat degenerate for $\ell=4$: The lattice norm of the short screening operator associated to the long root is an even integer (one might say it is still \emph{bosonic}) and it coincides with the respective long screening
$$({\alpha_1}^\oshort,{\alpha_1}^\oshort)=(\alpha_1/\sqrt{2},\alpha_1/\sqrt{2})=2
\qquad {\alpha_1}^\oshort=\alpha_1/\sqrt{2}=\sqrt{2}\alpha_1^\vee=\alpha_1^\olong$$
In particular by Corollary \ref{cor_NicholsRelations} no power of $\zem_{\alpha_1^\oshort}$ vanishes and, more severely, because of $p\not>4$ the Corollary does not prove Nichols algebra relations form them. To be precise, the Nichols algebra generated by the short screening operators of all long roots of $B_n,\ell=4$ is apparently trivial (commutative), while in fact they generate a Lie algebra $D_n$ .     \\

Thus, instead of using these short screening momenta $\alpha_1^\oshort,\alpha_2^\oshort$ for short screening operators, we have to use a set of simple roots for the subsystem of short roots $\alpha_2/\sqrt{2},(\alpha_1+\alpha_2)/\sqrt{2}$, which is a root subsystem of type $A_1\times A_1$. So we define:

$$\zem_1 := \zem_{- \frac{\alpha_1 + \alpha_2}{\sqrt{2}}}
\qquad \zem_2 := \zem_{- \frac{\alpha_2}{\sqrt{2}}}$$
 Their norm is an odd integer (they are \emph{fermionic}) and Corollary \ref{cor_NicholsRelations} reads:
$$(\zem_1)^2=(\zem_2)^2=0\qquad [\zem_1,\zem_2]=0$$ 
Similarly for $B_n,\ell=4$ we get $n$ commuting screening operators corresponding to the $A_1^n$ root system of short roots inside $B_n$. This degeneracy behaviour corresponds precisely to a similar degeneracy on the side of the quantum group of divided powers, as explained in Section \ref{sec_quantumgroup}. 

\subsection{Screening operators, restriced modules}\label{section_B2Kernels}
We begin by exemplary calculating the formula in Definition \ref{def_screening} the action of the screening operators on the \emph{Blue} module $\V_{[0]}$ and \emph{Green} module  $\V_{[{\lambda_1}/\sqrt{2}]}$ 
on the groundstates (which have conformal dimension $0$) and on the next level (conformal dimension $1$), which contains new pure exponential and differential polynomials times the groundstates. We will thereby determine the first layers of the intersection of the kernel of both screenings, which we denote in analogy to the case $A_1=B_1$
$$
\W=\Lambda(1):=\bigcap_{i=1,2}\Ker_{\V_{[0]}}\zem_i 
\qquad \Pi(1):=\bigcap_{i=1,2}\Ker_{\V_{[{\lambda_1}/\sqrt{2}]}}\zem_i
$$
On the Blue module $\V_{[0]}=\V_{\Lambda^\olong}$:
\begin{enumerate}[a)]
\item The screening operators applied to the groundstates elements (conformal dim. $0$)
	\begin{itemize} 
		\item $\zem_2 (1) = 0$
		\item $\zem_1 (1) = 0$
		\item $\zem_2 (e^{\phi_{\frac{\alpha_1}{ \sqrt{2}} + \sqrt{2}\alpha_2}}) = e^{\phi_{\frac{(\alpha_1 + \alpha_2)}{ \sqrt{2}}}}$
		\item $\zem_1 (e^{\phi_{\frac{\alpha_1}{\sqrt{2}} + \sqrt{2}\alpha_2}}) = e^{\phi_{\frac{\alpha_2}{ \sqrt{2}}}}$
	\end{itemize}
	\item The screening operators applied to the groundstate times a degree $1$ differential polynomial (conformal dim. $1$) is
	\begin{itemize} 
		\item $\zem_2 (\partial \phi_{-\alpha_2/\sqrt{2}}) =  e^{\phi_{-\frac{\alpha_2}{ \sqrt{2}}}}$
		\item $\zem_1 (\partial \phi_{-\alpha_2/\sqrt{2}}) = 0$
		\item $\zem_2 (\partial \phi_{-(\alpha_1 + \alpha_2)/\sqrt{2}}) = 0$
		\item $\zem_1 (\partial \phi_{-(\alpha_1 + \alpha_2)/\sqrt{2}}) = e^{\phi_{-\frac{(\alpha_1 + \alpha_2)}{\sqrt{2}}}}$
		\item $\zem_2 ( \partial \phi_{(\alpha_1 + \alpha_2)/\sqrt{2}}e^{\phi_{\sqrt{2}\alpha_1 + \frac{\alpha_2}{ \sqrt{2}}}}) = \partial \phi_{(\alpha_1 + \alpha_2)/\sqrt{2}}e^{\phi_{\frac{(\alpha_1 + \alpha_2)}{\sqrt{2}}}} $
		\item $\zem_1 ( \partial \phi_{(\alpha_1 + \alpha_2)/\sqrt{2}}e^{\phi_{\sqrt{2}\alpha_1 + \frac{\alpha_2}{ \sqrt{2}}}}) = 0$
		\item $\zem_2 (\partial \phi_{-\alpha_2/\sqrt{2}}e^{\phi_{\sqrt{2}\alpha_1 + \frac{\alpha_2}{ \sqrt{2}}}}) = 0 $
		\item $\zem_1 (\partial \phi_{-\alpha_2/\sqrt{2}}e^{\phi_{\sqrt{2}\alpha_1 + \frac{\alpha_2}{ \sqrt{2}}}}) = \partial \phi_{\alpha_2/\sqrt{2}}e^{\phi_{\frac{\alpha_2}{ \sqrt{2}}}}$ 
	\end{itemize}
	\item The screening operators applied to the next pure exponentials (conformal dim. $1$) is
	\begin{itemize} 
		\item $\zem_2 (e^{\phi_{-\frac{\alpha_1}{\sqrt{2}}}}) =  e^{\phi_{-\frac{(\alpha_1 + \alpha_2)}{\sqrt{2}}}}$
		\item $\zem_1 (e^{\phi_{-\frac{\alpha_1}{\sqrt{2}}}}) = 0$
		\item $\zem_2 (e^{\phi_{\frac{\alpha_1}{ \sqrt{2}}}}) = 0$
		\item $\zem_1 (e^{\phi_{\frac{\alpha_1}{ \sqrt{2}}}}) = e^{\phi_{-\frac{\alpha_2}{\sqrt{2}}}}$
		\item $\zem_2 (e^{\phi_{\sqrt{2}(\alpha_1 +\alpha_2)}}) = 0 $
		\item $\zem_1 ( e^{\phi_{\sqrt{2}(\alpha_1 +\alpha_2)}}) = e^{\phi_{\frac{(\alpha_1 + \alpha_2)}{\sqrt{2}}}}$
		\item $\zem_2 (e^{\phi_{\sqrt{2}\alpha_2}})) = e^{\phi_{\frac{\alpha_2}{\sqrt{2}}}} $
		\item $\zem_1 (e^{\phi_{\sqrt{2}\alpha_2}})) = 0$
	\end{itemize}
\end{enumerate}
On the Green module $\V_{[{\lambda_1}/\sqrt{2}]}$:
\begin{enumerate}[a)]
		\item  The screening operators applied to the groundstates elements  (conformal dim. $0$) is
	\begin{itemize} 
		\item $\zem_2 (e^{\phi_{\frac{(\alpha_1 + \alpha_2)}{\sqrt{2}}}}) = 0$
		\item $\zem_1 (e^{\phi_{\frac{(\alpha_1 + \alpha_2)}{\sqrt{2}}}}) = 1$
		\item $\zem_2 (e^{\phi_{\frac{\alpha_2}{\sqrt{2}}}}) = 1$
		\item $\zem_1 (e^{\phi_{\frac{\alpha_2}{\sqrt{2}}}}) = 0$
	\end{itemize}
	\item The screening operators applied to the groundstate times a degree $1$ differential polynomial (conformal dim. $1$) is
	\begin{itemize} 
		\item $\zem_2 ((\partial \phi_{\alpha_1/\sqrt{2}} + \partial \phi_{\alpha_2/\sqrt{2}})e^{\phi_{\frac{(\alpha_1 + \alpha_2)}{\sqrt{2}}}}) =  0$
		\item $\zem_1 ((\partial \phi_{\alpha_1/\sqrt{2}} + \partial \phi_{\alpha_2/\sqrt{2}})e^{\phi_{\frac{(\alpha_1 + \alpha_2)}{\sqrt{2}}}}) = 0$
		\item $\zem_2 (\partial \phi_{\alpha_2/\sqrt{2}}e^{\phi_{\frac{\alpha_2}{\sqrt{2}}}}) = 0$
		\item $\zem_1 (\partial \phi_{\alpha_2/\sqrt{2}}e^{\phi_{\frac{\alpha_2}{\sqrt{2}}}}) = 0$
		\item $\zem_2 (\partial \phi_{\alpha_2/\sqrt{2}}e^{\phi_{\frac{(\alpha_1 + \alpha_2)}{\sqrt{2}}}}) =  e^{\phi_{\frac{\alpha_1}{\sqrt{2}}}}$
		\item $\zem_1 ( \partial \phi_{\alpha_2/\sqrt{2}}e^{\phi_{\frac{(\alpha_1 + \alpha_2)}{\sqrt{2}}}}) = \partial \phi_{\alpha_1 /\sqrt{2}}$
		\item $\zem_2 (\partial \phi_{-(\alpha_1 + \alpha_2)/\sqrt{2}}e^{\phi_{\frac{\alpha_2}{\sqrt{2}}}}) =  \partial \phi_{-(\alpha_1 + \alpha_2)/\sqrt{2}}$
		\item $\zem_1 (\partial \phi_{-(\alpha_1 + \alpha_2)/\sqrt{2}}e^{\phi_{\frac{\alpha_2}{\sqrt{2}}}}) = e^{\phi_{-\frac{\alpha_1}{\sqrt{2}}}}$
	\end{itemize}
	\item The screening operators applied to the next pure exponentials (conformal dim. $1$) is
	\begin{itemize} 
		\item $\zem_2 (e^{\phi_{-\frac{(\alpha_1 + \alpha_2)}{\sqrt{2}}}}) =  0$
		\item $\zem_1 (e^{\phi_{\frac{(\alpha_1 + \alpha_2)}{\sqrt{2}}}}) = 0$
		\item $\zem_2 (e^{\phi_{-\frac{\alpha_2}{\sqrt{2}}}}) = 0$
		\item $\zem_1 (e^{\phi_{-\frac{\alpha_2}{\sqrt{2}}}}) = 0$
		\item $\zem_2 (e^{\phi_{\sqrt{2} \alpha_1 + \frac{3}{\sqrt{2}} \alpha_2}}) = e^{\phi_{\sqrt{2}(\alpha_1 +\alpha_2)}} $
		\item $\zem_1 ( e^{\phi_{\sqrt{2} \alpha_1 + \frac{3}{\sqrt{2}} \alpha_2}}) = e^{\phi_{\frac{\alpha_1}{\sqrt{2}}  + \sqrt{2} \alpha_2}}$
		\item $\zem_2 (e^{\phi_{\frac{\alpha_1}{\sqrt{2}} + \frac{3}{\sqrt{2}} \alpha_2}})) = e^{\phi_{\frac{\alpha_1}{\sqrt{2}}  + \sqrt{2} \alpha_2}} $
		\item $\zem_1 (e^{\phi_{\frac{\alpha_1}{\sqrt{2}} + \frac{3}{\sqrt{2}} \alpha_2}})) = e^{\phi_{\sqrt{2} \alpha_2}}$
	\end{itemize}
\end{enumerate}

\begin{figure}[htbp]
	\centering
	\label{fig:KerB2}
\end{figure}

On the Center module $\V_{[\lambda_2/\sqrt{2}]}$ and the Steinberg module  $\V_{[(\lambda_1+\lambda_2)/\sqrt{2}]}$ again by Corollary \ref{cor_Weylaction} the suitable short screening operators are $\zem_i^0=\id$ resp. $\zem_i^2=0$, so trivial or full kernel. Here we have $1$ resp. $4$ groundstates in conformal dimension $-1/4$ resp. $+1/4$. On the next level $+3/4$ resp. $+5/4$ the Center module has a basis of $2$ degree $1$ differential polynomials below the single groundstate and $4$ new pure exponentials, the Steinberg module has a basis of $2\cdot 4$ degree $1$ differential polynomials below the $4$ groundstates but no new pure exponentials (the next exponentials have by geometry already $+9/4$). Altogether we have for the first two $L_0$-layers of the kernel of the suitable short screening operators 

$$\begin{array}{l|cccc} 
 \mbox{Module}  & \V_{[\lambda]} &  \Ker(\zem_i)^k & \bigcap_{i=1,2} \Ker(\zem_i)^k  \\
\hline
\mbox{\emph{Blue}},\;k=1 &  2+8+\cdots & 1+4+\cdots& 1+0+\cdots   \\ 
\mbox{\emph{Green}},\;k=1 & 2+8+\cdots & 1+4+\cdots & 0+4+\cdots   \\
\mbox{\emph{Center}}& 1+6+\cdots & 0+0+\cdots & 0+0+\cdots  \\
\mbox{\emph{Steinberg}}\;k=2 & 4+8+\cdots & 4+8+\cdots & 4+8+\cdots \\
 \end{array}$$

 We again define as our new VOA the kernel of the short screening operators
 $$\mathcal{W}=\bigcap_{i=1,2}\Ker_{\V_{[0]}}\zem_i$$ 
 and study how the $4$ modules $\V_{[0]},\V_{[\lambda_1/\sqrt{2}]},\V_{[\lambda_2/\sqrt{2}]},\V_{[(\lambda_1+\lambda_2)/\sqrt{2}]}$ restricted to $\W$ decompose into kernels and cokernels of the short screening operators. For the Blue module and Green module we have submodules
 $$
 \begin{array}{lllll}
		  &			       	& \Ker \zem_1	&				&\\	
		  &\rotatebox{45}{$\supset$}   	& 		& \rotatebox{-45}{$\supset$} 	&\\
  \hspace{1.5cm}\V_{[0]}  &		& 		&				& \bigcap_{i=1,2} \Ker \zem_i=\W=:\Lambda(1) \\
		  &\rotatebox{-45}{$\supset$}   & 		& \rotatebox{45}{$\supset$} 	&\\
  		  &			       	& \Ker \zem_2	&				&\\	
 \end{array}
 $$ 
 $$
 \begin{array}{lllll}
		  &			       	& \Ker \zem_1	&				&\\	
		  &\rotatebox{45}{$\supset$}   	& 		& \rotatebox{-45}{$\supset$} 	&\\
  \V_{[\lambda_1/\sqrt{2}]}  &			& 		&				& \bigcap_{i=1,2} \Ker \zem_i=:\Pi(1) \\
		  &\rotatebox{-45}{$\supset$}   & 		& \rotatebox{45}{$\supset$} 	&\\
  		  &			       	& \Ker \zem_2	&				&\\	
 \end{array}
 $$ 
 and since these two short screenings commute we know the isomorphisms between kernels and cokernels and thus have composition series' (for $i$ arbitrary) with irreducible quotients indicated by underbraces:
 \begin{align*}
 \underbrace{\V_{[0]}\;\supset\;\sum_{i=1,2}\Ker}_{\Lambda(1)}
 \underbrace{\zem_i\;\supset\;\vphantom{\sum_{i=1,2}} \Ker}_{\Pi(1)}
 \underbrace{\zem_i\;\supset\; \bigcap_{i=1,2} \Ker}_{\Lambda(1)}
 \underbrace{\zem_i\;\supset\; \vphantom{\bigcap_{i=1,2}} \{0\}}_{\Pi(1)}\\
 \underbrace{\V_{[\lambda_1/\sqrt{p}]}\;\supset\;\sum_{i=1,2}\Ker}_{\Pi(1)}
 \underbrace{\zem_i\;\supset\;\vphantom{\sum_{i=1,2}} \Ker}_{\Lambda(1)}
 \underbrace{\zem_i\;\supset\; \bigcap_{i=1,2} \Ker}_{\Pi(1)}
 \underbrace{\zem_i\;\supset\; \vphantom{\bigcap_{i=1,2}} \{0\}}_{\Lambda(1)}
 \end{align*}
(we remark to the familiar reader that despite the obvious similarity this is not connected to the structure of the projective module in the case $A_1$; on the present modules $L_0$ acts still diagonal).\\
 
On the other hand the Center module and Steinberg module stay irreducible

$$\Lambda(2):=\V_{[\lambda_2/\sqrt{2}]}\qquad  \Pi(2):=\V_{[(\lambda_1+\lambda_2)/\sqrt{2}]}$$
 

\section{General case $B_n,\ell=4$} \label{B_n}
\subsection{Lattices}
We now consider the general case $\g=\so_{2n+1}$ with root system $B_n$ at $\ell=2p=4$. A base of the root lattice $\Lambda_R$ the set of $n$ simple roots $\{\alpha_1, \ldots, \alpha_n\}$ with $\alpha_n$ the unique short simple root and with killing form:
$$
(\alpha_i,\alpha_j)=
\begin{pmatrix}
  4 & -2 & 0 & \dots & 0\\
  -2 & 4 &  &  & \vdots\\ 
  0 &  & \ddots &  & 0\\
	\vdots &  &  & 4 & -2 \\
   0 & \dots & 0 & -2 & 2
 \end{pmatrix}$$
The coroots are then $\alpha_i^\vee = \alpha_i/2$ for $i=1 \ldots n-1$ and $\alpha_n^\vee = \alpha_n$. A streightforward calculation gives $Q$ determining the Virasoro action with central charge $c$ 

\begin{align*}
 \rho_{\g} &= \frac{1}{2} \left( \sum\limits_{j=1}^n j(2n-j)\alpha_j\right) \\
 \rho_{\g} &= \frac{1}{2} \left( \sum\limits_{j=1}^n j(2n-j)\alpha_j\right) \\
 Q &= \frac{1}{\sqrt{p}} (p {\rho_\g}^\vee - \rho_\g)  = \frac{1}{2\sqrt{2}} \sum\limits_{j=1}^n j\alpha_j  \\
 c &= \rank-12(Q,Q) = n-12\frac{n}{4}=-2n
\end{align*}
The \textit{short screening lattice} with its base is in this special case an odd integral lattice
$$\Lambda^\oshort = \frac{1}{\sqrt{2}} \Lambda_R , 
\qquad \{\alpha_1^\oshort,\ldots \alpha_n^\oshort\}=\{-\frac{1}{\sqrt{2}}\alpha_1, \ldots, -\frac{1}{\sqrt{2}} \alpha_n \}$$
The \textit{long screening lattice} with its base is the even integral lattice 
$$\Lambda^\olong = \sqrt{2} {\Lambda_R}^\vee , 
\qquad \{\alpha_1^\olong,\ldots \alpha_n^\olong\}= \left\{ \frac{\alpha_1}{\sqrt{2}}, \ldots, \frac{\alpha_{n-1}}{\sqrt{2}}, \sqrt{2} \alpha_n\right\} $$
The dual of the \textit{long screening lattice} with its base in terms of fundamental weights is
		\begin{align*}
		(\Lambda^\olong)^* &= \frac{1}{\sqrt{2}} \Lambda_W  
		\qquad \{{\lambda_1}/\sqrt{2}, \ldots, {\lambda_n}/\sqrt{2}\} = \{ \frac{\alpha_1}{\sqrt{2}} + \ldots + \frac{\alpha_n}{\sqrt{2}} , \; \frac{\alpha_1}{\sqrt{2}} + \frac{2}{\sqrt{2}}\left(\alpha_2 + \ldots + \alpha_n\right), \ldots \\
		&\ldots, \frac{\alpha_1}{\sqrt{2}} + \frac{2}{\sqrt{2}}\alpha_2 + \ldots + \frac{i-1}{\sqrt{2}}\alpha_{i-1} + \frac{i}{\sqrt{2}} \left(\alpha_i + \ldots + \alpha_n\right), \; \frac{1}{2\sqrt{2}}\left(\alpha_1 + 2\alpha_2 + \ldots + n\alpha_n \right)\}
		\end{align*}
Let us mention that as in the $B_2$ case (\ref{B_2}) we have ${\lambda_n}/\sqrt{2} = Q$.\\

\noindent
We compute the number of representations of $\V_\Lambda^\olong$ using formula (\ref{formula_numberReps}):
\begin{align*}
 \left| (\Lambda^\olong)^* / \Lambda^\olong\right|
 &=|{(\Lambda^{\olong})}^*/\Lambda^{\oshort}|\cdot |{(\Lambda^{\oshort})}/\Lambda^{\olong}|\\
 &=|\frac{1}{\sqrt{2}} \Lambda_W / \frac{1}{\sqrt{2}}\Lambda_R |\cdot |\frac{1}{\sqrt{2}}\Lambda_R / \sqrt{2}\Lambda_R^\vee|\\
 &=|\Lambda_W / \Lambda_R|\cdot \prod_{i=1}^{\rank} \frac{\ell}{(\alpha_i,\alpha_i)}
 =2\cdot 2
\end{align*}
Since $\Lambda^\oshort / \Lambda^\olong = \left\{ 0 + \Lambda^\olong, \frac{\alpha_n}{\sqrt{2}} + \Lambda^\olong \right\}$, our 4 modules are given by the cosets
\[ [0], \qquad [\frac{\alpha_n}{\sqrt{2}}], \qquad [Q], \qquad [Q + \frac{\alpha_n}{\sqrt{2}}]
\] 
respectively the \emph{Blue}, \emph{Green}, \emph{Center} and \emph{Steinberg} module.

\subsection{Groundstates}

By looking at the geometry of the four $\Lambda^\olong=\sqrt{2}\Lambda_R^\vee$ cosets in $(\Lambda^\olong)^*=\sqrt{1}{\sqrt{2}}\Lambda_W$ we solve the closest-neighbour problem for the point $Q$ to determine the groundstates of our four modules. Also for later use, the result is stated in terms of a basis of a larger lattice  with orthogonal basis given by \emph{short} screenings for short roots 
$$\Lambda^\olong\subset (\frac{1}{\sqrt{2}}A_1)^n \subset \Lambda^\oshort  
\qquad \alpha_{k\ldots n}^\oshort=\sum_{i=k}^n \alpha_i/\sqrt{2}$$


With the previous calculation for $(Q,Q)=n/4$ and picking representatives we also calculated the conformal dimension of the respective groundstates:

$$\begin{array}{l|llll}
 \mbox{Module }\V_{[\lambda]} & \#\mbox{Groundstates} & \mbox{Conformal Dim} & \mbox{Groundstates elements }\\
\hline
\emph{Blue} & 2^{n-1} & 0 & Q + \frac{1}{2}\sum\limits_{k =1}^n \epsilon_k {\alpha_{k \ldots n}}^\oshort $ with $\prod_i{\epsilon_i} = +1 \\ 
\emph{Green} & 2^{n-1} & 0 & Q + \frac{1}{2}\sum\limits_{k =1}^n \epsilon_k {\alpha_{k \ldots n}}^\oshort$ with $\prod_i{\epsilon_i} = -1\\
\emph{Center} & 1 & -\frac{n}{8} &  Q\\
\vphantom{x^{x^{x^{x^{x^{x^x}}}}}}
\emph{Steinberg} & 2n & -\frac{n}{8}+\frac{1}{2} & Q \pm {\alpha_{k \ldots n}}^\oshort
\vphantom{x^{x^{x^{x^{x^{x^x}}}}}}
\end{array}
$$

\subsection{Screening operators, restriced modules}

In the general case $B_n,\ell=4$ we have degeneracies similar to $B_2,\ell=4$, because short screening operators of long roots are equal to long screening operators, and have even integral norm (bosonic). So again we define our relevant short screening operators from the short roots in $B_n$, which form an $A_1^n$ root system with orthogonal basis $\alpha_{i\ldots n}^\oshort$ introduced in the previous section.  Their norm is an odd integer (fermionic) and Corollary \ref{cor_NicholsRelations} reads
$$\zem_{\alpha_{i\ldots n}^\oshort}:
\qquad (\zem_{\alpha_{i\ldots n}^\oshort})^2=0
\qquad [\zem_{\alpha_{i\ldots n}^\oshort},\zem_{\alpha_{j\ldots n}^\oshort}]=0$$ 
Compare this again to the degeneracy in the associated quantum group of divided powers as explained in Section \ref{sec_quantumgroup}. \\

Similar to the case $B_2$, for $B_n$ the four $\V_{\Lambda^\olong}$ modules restricted to the kernel of the commuting short screenings
$$\W=\bigcap_{i=1,\ldots n}\ker_{\V_{[0]}}\zem_{\alpha_{i\ldots n}}$$
decompose as follows:
\begin{itemize}
 \item The Blue module (vacuum module) $\V_{[0]}=\V_{\Lambda^\olong}$ has a decomposition series of $2^{n-1}$ modules $\Lambda(1)$ and $2^{n-1}$ modules $\Pi(1)$, where the socle is the vacuum submodule 
 $$\Lambda(1)=\bigcap_{i=1,\ldots n}\ker_{\V_{[0]}}\zem_{\alpha_{i\ldots n}}=\W$$
 \item The Green module has a decomposition series of $2^{n-1}$ modules $\Pi(1)$ and $2^{n-1}$ modules $\Lambda(1)$, where the socle is the submodule 
 $$\Pi(1)=\bigcap_{i=1,\ldots n}\ker_{\V_{[\alpha_n/\sqrt{2}]}}\zem_{\alpha_{i\ldots n}}$$
 \item The Center module $\V_{[Q]}$ stays irreducible, call it $\Lambda(2)$.
 \item The Steinberg module stays irreducible, call it $\Pi(2)$.
\end{itemize}

\section{Equivalence to even part of $n$ pairs of symplectic fermions}

Our main result below is that the VOA $\W$ defined as kernel of the short screenings, here $\zem_{\alpha_{k\ldots n}^\oshort}$, is as vertex algebra isomorphic to a well-known vertex algebra: The even part of $n$ pairs of symplectic fermions.   

\subsection{The vertex algebra of symplectic fermions}

The vertex algebra of a single pair of symplectic fermions $\V_{\SF}=\V_{\SF_1}$ is a super-vertex algebra with central charge $-2$ introduced by \cite{Kausch,Abe07}, our exposition follows \cite{DR16}: In vertex algebra language, the symplectic fermions are generated by fields fermionic fields $\psi,\psi^*$ with OPE
\begin{align*}
\psi(z)\psi^*(w)
&=\frac{1}{(z-w)^2}+\frac{0}{z-w}+\cdots\\
\psi^*(z)\psi(w)
&=\frac{-1}{(z-w)^2}+\frac{0}{z-w}+\cdots
\end{align*}
so the mode operators fulfill the following anticommutators
$$\psi(z)=\sum_{n\in\Z}\psi_{-n-1}z^n
\qquad \psi^*(z)=\sum_{n\in\Z}\psi^*_{-n-1}z^n
\qquad [\psi_n,\psi^*_m]_+=n\delta_{n+m=0}\;\id$$
This super-vertex algebra can equivalently be described as free super-bosons, meaning the vacuum module of the affine super Lie algebra $\hat{\mathfrak{h}}$ associated to the abelian super Lie algebra $\mathfrak{h}$, which is a purely odd nondegenerate symplectic vector space of dimension $2$.\\

Similarly one defines the super-vertex algebra of $n$-pairs of symplectic fermions \cite{Abe07}
$$\V_{\SF_n}=\left(\V_{\SF_1}\right)^n$$ 

We then consider the vertex algebra which is the even part of this super-vertex algebra
$$\V_{\SF_n^{even}}=\left(\V_{\SF_1}\right)^{n,\;even}\supsetneq \left(\V_{\SF_1^{even}}\right)^n$$ 

\subsection{Representation theory}\label{sec_SFRep}

It is known from \cite{Kausch,GK96,Abe07} that $\V_{\SF_n^{even}}$ is a logarithmic CFT with four irreducible representations
$${\chi_1}^{\mathcal{SF}_n^{even}},
\qquad {\chi_2}^{\mathcal{SF}_n^{even}},
\qquad {\chi_3}^{\mathcal{SF}_n^{even}},
\qquad {\chi_4}^{\mathcal{SF}_n^{even}}
$$
with groundstate conformal dimensions $0,1,-\frac{n}{8},-\frac{n}{8}+\frac{1}{2}$.

 In this article we only prove statements about vertex algebras and abelian categories. The unfamiliar reader should be made aware, that a celebrated result of Huang, Lepowsky, Zhang  \cite{HLZ10} shows that for any vertex algebra (admitting suitable finiteness conditions) the representation category can be canonically endowed with a tensor product, an associator and a braiding; conjecturally it is always a modular tensor category in the non-semisimple sense. However, this construction is rather involved and indirect, so it is hard to prove facts about this structure, e.g. an equivalence to a known modular tensor category. 
 
An important conjecture about symplectic fermions is that  

\begin{conjecture}[\cite{DR16} Conjecture 1.3]\label{con_SFIngo}
 $\V_{\SF_n^{even}}\md\mod$ is as modular tensor category equivalent to the abstract and explicit braided monoidal category $\SF_n$ defined in \cite{Runkel12} on $\mathcal{C}_0\oplus\mathcal{C}_1$. In particular as an abelian category it is equivalent to 
 $$\V_{\SF_n^{even}}\md\mod =\mathcal{C}_0+\mathcal{C}_1\cong \C[E_1,\ldots E_n,F_1,\ldots F_n](\mathrm{SVect})\md\mod \; \oplus \;  \mathrm{SVect} $$
with the \emph{Grassmaniam} polynomial algebra in the category of super-vector spaces generated by odd generators $E_i,F_i$ that anticommute. Under this equivalence the representations $\chi_3,\chi_4$ are $\C^{1|0},\C^{0|1}$ in the semisimple category $\mathcal{C}_1=\mathrm{SVect}$ and $\chi_1,\chi_2$ are $\C^{1|0},\C^{0|1}$ in $\mathcal{C}_0$ with zero-action of all $E_i,F_i$. \\
\end{conjecture}

For $n=1$ the structure of the abelian category coincides what has been proven in \cite{FGSTsl2}. The author would like to point to the obvious fact, that if the full conjecture is known for $n=1$, then one can extend $(\V_{\SF_1^{even}})^n$ to $\V_{\SF_n^{even}}$ as we do in the last section. On the other hand, the (a-priori unknown) braided tensor structure on $\V_{\SF_n^{even}}\md\mod$ has to be compatible with the known modular tensor category structure on the larger algebra $\V_{\Lambda}\md\mod$ for $\Lambda=\sqrt{2}C_n$ as we show. However, the restriction functor is not multiplicative (the tensor product in $\V_{\SF_n^{even}}\md\mod$ is larger, to be precise a projective cover), so we probably only have a lax monoidal functors, but the information might be enough to pin down the braided monoidal structure.

\subsection{Graded dimension} 

The following definition makes sense for every graded vector space and is crucial for the theory of vertex algebra modules. We use therefore the respective conventions:
\begin{definition}
Let $\V$ be a vertex algebra with a Virasoro action of central charge $c$, and let $\mathcal{M}$ be a vertex algebra module with $L_0$-grading $\mathcal{M}=\bigoplus_m \mathcal{M}_m$  such that the layers are finite-dimensional and the $L_0$-eigenvalues are discrete in $\Q$. Then we define the formal power series:
$$\dim(\mathcal{M})(t):= t^{-\frac{c}{24}} \sum_m dim V_m t^m$$
\end{definition}
\begin{remark}
 The unfamiliar reader should be made aware, that for vertex algebras (admitting suitable finiteness conditions) for which the representation category is semisimple, the graded dimensions (and other graded character values) on the irreducible modules piece together to a vector-valued modular form. This is the deeper reason behind the appearence of $\eta$- and $\Theta$-functions in what follows.
\end{remark}

\begin{example}
  The Virasoro module $\mathcal{M}_\lambda=\{u\exp{\phi_\lambda}\}$ over $\V_\Lambda$ with central charge $c(Q)$ has groundstate $\exp{\phi_\lambda}$ with groundstate conformal dimension 
  \begin{align*}
   c
   &=\rank-12(Q,Q)\\
   h(\lambda)
   &=\frac{1}{2}(\lambda-Q, \lambda-Q)-\frac{1}{2}(Q,Q)
  \end{align*}
  Let first $\rank(\Lambda)=1$, then the $n$-th layer has conformal dimension $h(\lambda)+n$ and the dimension of the $n$-th layer is the number of partitions $p(n)$. Alternatively, it is the polynomial ring in the variables $\partial^{k}\phi$ of degree $k>0$. Hence the graded character is the famous Dedekind $\eta$-function alternatively in the sum- or product-expansion
  \begin{align*}
  \frac{1}{\eta(t)}
  &:=t^{-\frac{1}{24}}\sum_{n\in\N_0} p(n)t^n
  =t^{-\frac{1}{24}}\prod_{k\in\N} \frac{1}{1-t^k}\\
  \dim(\mathcal{M}_\lambda)(t)
    &=\frac{t^{-\frac{c}{24}+h(\lambda)+\frac{1}{24}}}{\eta(t)}
    =\frac{t^{\frac{1}{2}(Q-\lambda,Q-\lambda)}}{\eta(t)}
  \end{align*}
  Similarly in arbitrary rank
  $$\dim(\mathcal{M}_\lambda)(t)=\frac{t^{\frac{1}{2}(Q-\lambda,Q-\lambda)}}{\eta(t)^{\rank}}$$
  The same formula holds for the induced and coinduced Virasoro modules, even though these are not isomorphic as Virasoro modules, but have the same decomposition series.
\end{example}
\begin{example}
  For the module $\V_{[\mu]}$ of the lattice vertex algebra $\V_\Lambda$ we have to sum over all $\lambda$ in the coset $[\mu]$. This gives the famous Jacobi theta function $\Theta_{\mu+\Lambda}(t)$ associated to the lattice coset $[\mu]=\mu+\Lambda$
  \begin{align*}
    \dim(\V_{[\mu]})
    &=\sum_{\lambda\in[\mu]} \frac{t^{\frac{1}{2}(Q-\lambda,Q-\lambda)}}{\eta(t)^{\rank}}
    =\frac{\Theta_{\mu-Q+\Lambda}(t)}{\eta(t)^{\rank}}
  \end{align*}
\end{example}

The graded characters of $\V_{\SF_n}$ are known, see e.g. \cite{DR16} Section 2.1: With the characters of the super-vertex algebra $\V_{\SF_n}=(\V_{\SF_1})^n$
\begin{align*}
\chi_{ns, +} &= \left(t^{\frac{1}{24}}\prod_{m = 1}^{\infty}(1 + t^m)\right)^{2n}\\
\chi_{r, +} &= \left(t^{-\frac{1}{48}}\prod_{m = 1}^{\infty}(1 + t^{m-\frac{1}{2}})\right)^{2n}
\end{align*} 
and the graded traces of the parity operator on the same modules
\begin{align*}
\chi_{ns, -} &= \left(t^{\frac{1}{24}}\prod_{m = 1}^{\infty}(1 - t^m)\right)^{2n} \\
\chi_{r, -} &= \left(t^{-\frac{1}{48}}\prod_{m = 1}^{\infty}(1 - t^{m-\frac{1}{2}})\right)^{2n}
\end{align*}
The graded dimensions of the four irreducible $\V_{\SF_n^{even}}$-modules are 
\begin{align*}
{\chi_1}^{\mathcal{SF}_n^{even}} &= \frac{1}{2}(\chi_{ns, +} + \chi_{ns, -})\\
{\chi_2}^{\mathcal{SF}_n^{even}} &= \frac{1}{2}(\chi_{ns, +} - \chi_{ns, -})\\
{\chi_3}^{\mathcal{SF}_n^{even}} &= \frac{1}{2}(\chi_{r, +} + \chi_{r, -}) \\
{\chi_4}^{\mathcal{SF}_n^{even}} &= \frac{1}{2}(\chi_{r, +} - \chi_{r, -})
\end{align*} 
One still sees that the modules actually come from decomposing the $\V_{\SF_n}$-modules, which are $n$-th tensor powers, into two composition factors 
$\chi_1,\chi_2$ resp $\chi_3,\chi_4$. This behaviour will also be visible in our $B_n,\ell=4$ model.

\subsection{Matching the characters}

We want to quickly calculate that the graded dimensions of the symplectic fermions $\V_{\SF_n^{even}}$ match the graded dimensions of our $\W$ for $\B_n,\ell=4$. This was our first hard evidence to the vertex algebra isomorphism we prove later. Since the matching of the graded dimension follows from this result, we only illustrate the case $n=2$, where $-\frac{c}{24}=+\frac{1}{6}$:\\

First we directly compare the graded dimensions to the first layers of $\Lambda(1),\Pi(1),\Lambda(2),\Pi(2)$, which we explicitly computed in Section \ref{section_B2Kernels}:
\begin{align*}
\dim(\Lambda(1))(t) &= t^{\frac{1}{6}+0}\left( 1 + 0t+\cdots\right)\\
\dim(\Pi(1))(t) &= t^{\frac{1}{6}+0}\left( 0 + 4t+\cdots\right)\\
\dim(\Lambda(2))(t) &= t^{\frac{1}{6}-\frac{1}{4}}\left( 1 + 6t+\cdots\right)\\
\dim(\Pi(2))(t) &= t^{\frac{1}{6}+\frac{1}{4}}\left( 4 + 8t+\cdots\right)
\end{align*} 
This matches nicely the characters $\chi_1,\chi_2,\chi_3,\chi_4$ in $\V_{\SF_n^{even}}$, which were defined as symmetrizations/antisymmetrizations of 
\begin{align*}
\chi_{ns, \pm} =t^{+\frac{1}{24}}(1\pm t)^4 (1\pm t^2)^4 \cdots\hspace{.30cm}
&= t^{+\frac{4}{24}}\left( 1 \pm 4t + 10 t^2 \pm \cdots \right)\\
\chi_{r, \pm} = t^{-\frac{1}{48}}(1 \pm t^{\frac{1}{2}})^{4}(1 \pm t^{\frac{3}{2}})^{4}\cdots
&=t^{-\frac{4}{48}}\left(1\pm 4t^{\frac{1}{2}}+6t\pm 8t^{\frac{3}{2}}+16t^{2}\right)
\end{align*}
We have also checked in general from the calculation of the lattice vertex algebra dimension in terms of Theta function that the graded dimensions of Blue and Green module $2^{n-1}\Lambda(1)+2^{n-1}\Pi(1)$ and Center and Steinberg module $\Lambda(2),\Pi(2)$ match the symplectic fermion character. Only for illustrative purposes and because it shows the way of our proof, we repeat this calculation for $B_1=A_1,\ell=4$ and check
$$\dim(\V_{[0]})=\dim(\Lambda(1))+\dim(\Pi)\stackrel{!}{=}\dim(\chi_1^{\SF_1^{even}})+\dim(\chi_2^{\SF_1^{even}})=\dim(\chi_{ns,+})$$
using the Jacobi triple product identity:

\begin{align*}
\chi_{[0]} &= \frac{\Theta_{[0-Q]}(t)}{\eta(t)} 
= \frac{\sum\limits_{\lambda \in [0]} t^{\frac{1}{2}( \lambda-Q, \lambda-Q)}}{t^{\frac{1}{24}}{\prod\limits_{m = 1}^{\infty} (1 - t^{m})}}
= t^{\frac{1}{8}}t^{-\frac{1}{24}}\frac{\sum\limits_{k \equiv 0 (4)} t^{\frac{1}{8}(k^2 - 2k)}}{{\prod\limits_{m = 1}^{\infty}(1 - t^{m})}} 
= t^{\frac{2}{24}} \frac{\sum\limits_{r} t^{2(r^2 - \frac{r}{2})}}{{\prod\limits_{m = 1}^{\infty}(1 - t^{m})}} \\
&= t^{\frac{2}{24}} \frac{\prod\limits_{m = 1}^{\infty} (1-t^{4m})(1+t^{-1} t^{4m-2})(1+t t^{4m-2})}{\prod\limits_{m = 1}^{\infty}(1 - t^{m})}    \\
&=t^{\frac{2}{24}} \prod\limits_{m = 1}^{\infty}(1+t^{m})(1+t^{2m})(1+t^{4m-3})(1+t^{4m-1})\\
&=\left(t^{\frac{1}{24}}\prod_{m = 1}^{\infty}(1 + t^m)\right)^{2}=\chi_{ns,+}
\end{align*}

\subsection{Isomorphism of symplectic fermions to $B_n,\ell=4$}

We have discussed thoroughly the output of the screening charge method for the datum $B_n,\ell=4$ and compared the character the even part of $n$ pairs of symplectic fermions. We now want to prove that these two vertex algebras are indeed isomorphic. The proof is rather simple: For $A_1=B_1$ this seems to be common knowledge, and since the super-vertex algebra just consists of $n$ copies, we only have to convince ourselves that the $B_n^\vee=C_n$-lattice, which is equal to the $D_n$-lattice, singles out the even part.

\begin{lemma}
  For the datum $\g=A_1,\ell=4$ the lattice $\Lambda^\oshort=\frac{1}{\sqrt{2}}\Lambda_R$ is an odd integral lattice. We have the following isomorphisms of super-vertex algebra and vertex algebras
  \begin{align*}
  \V_{\SF_1}
   &=\ker_{\V_{\Lambda^\oshort}}\zem_{-\alpha/\sqrt{2}}\\
   \V_{\SF_1^{even}}
   &\cong
   \W_{A_1,\ell=4}
  \end{align*}
\end{lemma}
\begin{proof}
  This seems to be common knowledge (see e.g. \cite{FGSTsl2}), but let us draw a quick proof:
  The isomorphisms sends the states
  \begin{align*}
    \psi
    &\longmapsto \exp{\phi_{-\alpha/\sqrt{2}}}\\
    \psi^*
    &\longmapsto \partial\exp{\phi_{+\alpha/\sqrt{2}}}
  \end{align*}
  This maps lands clearly in $\ker_{\zem_{-\alpha/\sqrt{2}}}$, and the OPE between the images is as in symplectic fermions, in our language:
  \begin{align*}
   \Y(\exp{\phi_{-\alpha/\sqrt{2}}}) 
   \partial\exp{\phi_{+\alpha/\sqrt{2}}}
   &=\sum_{k\geq 0} z^{-1}z^{k}
   \partial\exp{\phi_{+\alpha/\sqrt{2}}}\frac{\partial^k}{k!}\exp{\phi_{-\alpha/\sqrt{2}}}
   +\sum_{k\geq 0} (+z^{-2})z^k
   \exp{\phi_{+\alpha/\sqrt{2}}}\frac{\partial^k}{k!}\exp{\phi_{-\alpha/\sqrt{2}}}\\
   &=\exp{0}z^{-2}+0\;z^{-1}+\cdots
  \end{align*}
  If we want to explicitly prove the defining relations of the mode operators
  $$\psi(z)=\sum_{n\in\Z} \Y(\psi)_n z^n
  \qquad \psi^*(z)=\sum_{n\in\Z}\Y(\psi^*)_n z^n
  \qquad [\Y(\psi)_{-1-n},\Y(\psi^*)_{-1-m}]_+=m\delta_{m+n=0}$$
  for the images $\Y(a)_{-1-n},\Y(b)_{-1-m}$ from the previous calculation of $\Y(a)b$, we invoke the associativity formula for integer OPE's (here of a super-vertex algebra)
  \begin{align*}
   [\Y(a)_{-1-n},\Y(b)_{-1-m}]_+
   &=\sum_{l\geq 0} {n \choose l} \Y(\Y(a)_{-l-1}b)_{l-(1+n+m)}\\
   &= {n \choose 0} \Y(0)_{-(1+n+m)}+{n \choose 1} \Y(e^0)_{1-(1+n+m)}\\
   &=n\delta_{n+m=0}\;\id
  \end{align*}
  Since the vacuum representation is irreducible, the vertex algebra homomorphism is injective. Then surjectivity follows from the matching of the graded dimensions calculated for both vaccum representations above.\\
  
  It is easy to see that the even subspace $\V_{\SF_1^{even}}$, i.e. differential polynomials in $\psi,\psi^*$ with an even number of factors in each monomial, maps precisely to the subspace with $\Lambda$-degrees in $\Lambda^\olong=2\Lambda^\oshort$ hence 
  $$
   \V_{\SF_1^{even}}
   \cong
   \ker_{\V_{\Lambda^\olong}}\zem_{-\alpha/\sqrt{2}}
   =\W_{A_1,\ell=4}$$
\end{proof}

\begin{corollary}
By tensoring $n$ copies we get
   \begin{align*}
  \V_{\SF_n}
   =\left(\ker_{\V_{\frac{1}{\sqrt{2}}\Lambda_{R,A_1}}}\zem_{-\alpha/\sqrt{2}}\right)^n
   &=\bigcap_{i=1,\ldots n}\ker_{\V_{\frac{1}{\sqrt{2}}\Lambda_{R,A_1^n}}}\zem_{-\alpha_i^{A_1^n}/\sqrt{2}}\\
  \end{align*}
\end{corollary}
The even part is again completely characterized by its lattice-degrees: They are the sublattice $\Lambda\subset \frac{1}{\sqrt{2}}\Lambda_{R,A_1^n}=\Z^n$
consisting of even sums of the basis elements, i.e. 
$$\Lambda=\left\{\sum_i x_i(\alpha_i^{A_1^n}/\sqrt{2}) \mid x_i\in\Z,\;\sum_ix_i\in 2\Z\right\}
\;\subset\;\Z^n$$
But this ''chess-board-lattice`` $\Lambda$ is precisely the $D_n$-root lattice; this embedding of the even part is already described in \cite{DR16} Thm. 3.2 and Cor. 3.3. Then the $D_n$-lattice is equal to the rescaled root lattice $\frac{1}{\sqrt{2}}C_n$, which is in turn the rescaled coroot lattice $\sqrt{2}B_n^\vee$. But this is precisely the lattice prescribed by the screening charge method for $B_n,\ell=4$ and  the images of the orthogonal basis are the dual short root
$$\Lambda=\sqrt{2}\Lambda_{R,B_n}^\vee=\Lambda^\olong_{B_n,\ell=4}
\qquad \alpha_i^{A_1^n}/\sqrt{2}=\alpha_{i\ldots n}^{B_n}/\sqrt{2}$$
which we prescribed as short screening momenta in the degenerate case $B_n,\ell=4$. Altogether this gives our final result

\begin{corollary}
 $$\V_{\SF_n^{even}}
 =\bigcap_{i=1,\ldots n}\ker_{\V_{{\sqrt{2}}\Lambda_{R,B_n}^\vee}}\zem_{-\alpha_{i\ldots n}^{B_n}}
   \cong
   \W_{B_n,\ell=4}$$
\end{corollary}

\section{The quantum group side for $B_n,\ell=4$}\label{sec_quantumgroup}

The most interesting part of the main Conjectures \ref{conj_Main} about the screening charge method for constructing LCFT is that for chosen $\g,\ell$ the representation category of the vertex algebra $\W$ constructed as kernel of the short screenings is equivalent as abelian category to the representation category of a small quantum group $u_q(\g)$ associated to this data. Moreover there should be a similar quasi-Hopf algebra with an $R$-matrix such that these are even equivalent as modular tensor categories.

A complication is that the Virasoro calculations on the CFT side demand $(\alpha_i,\alpha_i)|\ell$ (see Section \ref{sec_rescaledLattices}) , while the usual restriction on the quantum group side is that they be relatively prime. Indeed, this leads to the appearence of dual root systems (as on the LCFT side) in Lustzig's infinite quantum group of divided powers, and it also falsifies the usual generic statement that small quantum groups have a canonical $R$-matrix. Moreover there are additional degeneracies for small values of $\ell$. The second author has studied these non-relatively-prime cases in \cite{LentCCM}, and the result in particular with respect to $B_n,\ell=4$ are as follows:\\

\subsection{The degenerate Lusztig quantum group of divided powers}\label{sec_degenerateLusztig}

To every datum $\g,\ell$ with $q$ a primitive $\ell$th root of unity, where $\ell$ is relatively prime to $(\alpha_i,\alpha_i)$ there exists an infinite-dimensional Hopf algebra called \emph{Lusztig quantum group of divided powers} $U_q^{\L}(\g)$, such that there is an exact sequence of Hopf algebras
$$u_q(\g)\longrightarrow U_q^{\L}(\g)\longrightarrow U(\g)$$
where $u_q(\g)$ is the small quantum group. In fact from a theoretic view $U_q^{\L}(\g)$ is the more natural object, which is defined for every datum by specializing an integral form of the quantum group over the field $\C(q)$. In \cite{LentCCM} the second author has worked out the cases where the divisibility condition is dropped. As turns out, the infinite quantum group again decomposes into a short exact sequence of Hopf algebras (at least on the level of Borel parts and up to a symmetric braiding, which we omit here for simplicity) 
$$u_q(\g^{(0)})^+\longrightarrow U_q^{\L}(\g)^+\longrightarrow U(\g^{(\ell)})^+$$
with possibly \emph{different Lie algebras} $\g^{(0)},\g^{(\ell)}$ depending on the divisibilities as follows:
\begin{center}
\begin{tabular}{l|ll|l|l}
& $\g\qquad$ & $\ell=\ord(q)$ & $\g^{(0)}\quad$ &
$\g^{(\ell)}\quad$ \\
\hline\hline
\textnormal{Trivial cases:}
& all & $\ell=1$ & $0$ & $\g$ \\
& all & $\ell=2$ & $0$ & $\g$ \\
\cline{2-5}
\multirow{4}{3cm}{\textnormal{Generic cases:}} 
& $ADE$ & $\ell\neq 1,2$ & $\g$ & $\g$  \\
&$B_n$ & $4\nmid \ell\neq 1,2$ & $B_n$ & $B_n$  \\
&$C_n$ & $4\nmid \ell\neq 1,2$ & $C_n$ & $C_n$  \\
&$F_4$ & $4\nmid \ell\neq 1,2$ & $F_4$ & $F_4$ \\ 
&$G_2$ & $3\nmid\ell\neq 1,2,4$ & $G_2$ & $G_2$\\
\cline{2-5}
\multirow{4}{3cm}{\textnormal{Duality cases:}$\quad$}
& $B_n$ & $4|\ell\neq 4$ & $B_n$ & $C_n$ \\
& $C_n$ & $4|\ell\neq 4$ & $C_n$ & $B_n$ \\
& $F_4$ & $4|\ell\neq 4$ & $F_4$ & $F_4$ \\
& $G_2$ & $3|\ell\neq 3,6$ & $G_2$ & $G_2$\\
\cline{2-5}
\multirow{4}{3.5cm}{\textnormal{Degenerate cases:}$\quad$}
& $B_n$ & $\ell=4$ & $A_1^{n}$ & $C_n$   \\      
& $C_n$ & $\ell=4$ & $D_n$ & $B_n$ \\
& $F_4$ & $\ell=4$ & $D_4$ & $F_4$  \\ 
& $G_2$ & $\ell=3,6$ & $A_2$ & $G_2$\\
\cline{2-5}
\textnormal{Exotic case:}
& $G_2$ & $\ell=4$ & $A_3$ & $G_2$ \\ 
\hline
\end{tabular}
\end{center}
In principle we see two effects: 
\begin{itemize}
\item If $\ell$ has common divisors with the $(\alpha_i,\alpha_i)/2$, then the Lie algebra contained in $U_q^{\L}(\g)$ changes to the dual Lie algebra $\g^{(\ell)}=\g^\vee$. This is because it is generated by divided powers according to the relations 
$$E_{\alpha_i}^{(\ell_{\alpha_i})},\qquad \ell_{\alpha}:=\ord(q^{(\alpha,\alpha)})$$ 
and in general these exponents do not have to coincide for differently long $\alpha$.
\item For small values of $\ell$, to be precise $\ord(q^{(\alpha_i,\alpha)})\leq \ell$, the small quantum group degenerates. The strict inequality only holds for $G_2,\ell=4$ with very exotic behaviour. In all other degenerate cases $\ord(q^{(\alpha_i,\alpha)})=\ell$ the small order of $\ell$ implies trivial braiding $q^{(\alpha_i,\alpha_i)}=1$ for the long roots. In effect the root vectors $E_{\alpha_i}$ have infinite order and are not contained in the finite-dimensional small quantum group $u_q(\g)$, but become the short roots in the dual Lie algebra $\g^{(\ell)}=\g^\vee$. The only root vectors remaining in the small quantum group $u_q(\g^{(0)})$ are the (possibly non-simple) root vectors $E_{\alpha}$ for short roots $\alpha$. 
\end{itemize}

To be more precise in the case relevant to this paper:
\begin{lemma}
Let $B_n,\ell=4$ with $\alpha_n$ the unique short simple root. Then the
infinite-dimensional quantum group $U_q^{\L}(\g)$ has algebra generators $K_{\alpha_i}^{\pm1}$ and 
$$E_{\alpha_n},\;E_{\alpha_{n-1}+\alpha_n},\;\ldots$$
$$E_{\alpha_1},\;E_{\alpha_2}\ldots E_{\alpha_{n-1}},E_{\alpha_{n-1}+2\alpha_2}$$
$$E_{\alpha_n}^{(2)},\;E_{\alpha_{n-1}+\alpha_n}^{(2)},\;\ldots$$
(and similarly for $F$) and has the following relations between the $E_\alpha^{(k)}$
\begin{itemize}
 \item The first line of generators (short roots) square to zero and mutually commute. They generate a small quantum group $u_q(A_1^n)$ at $\ell=4$. 
 \item The second line of generators (long roots) have infinite order and generate the Lie algebra $D_n$, which is the subsystem of long roots of $B_n$ and the subset of short roots for $C_n$.
 \item The third line of generators (divided powers of short roots) have infinite order and mutually commute, the generate the Lie algebra $A_1^n$, which is the subsystem of long roots for $C_n$.
 \item The second and third line of generators together generate the dual Lie algebra $C_n$.
 \item This Lie algebra $C_n=\mathfrak{sp}_{2n}$ acts by adjoint action on $A_1^n$ as on a symplectic vector space of dimension $2n$.
\end{itemize}
\end{lemma}

This is obviousely the same behaviour as the short and long screenings in our LCFT above, in particular we recover the the global action of the symplectic Lie algebra by long screenings. Moreover we shall see that we get a category equivalence between the representation categories of $\W_{\B_n,\ell=4}=\V_{\SF_n^{even}}$ to the algebra $u_q(A_1^n)$ with suitable Cartan part

\subsection{Category equivalence quantum group to $B_n,\ell=4$}

In a recent work \cite{FGR17b} on $\SF_n$, extending previous work \cite{GR,FGR17a} for $n=1$, it is proven that the explicit abstract modular tensor category $\SF_n^{even}$ is equivalent to the representation category of a quasi-Hopf algebra $\tilde{u}$ with quantum Borel part of type $A_1^n$ and Cartan part $\tilde{u}=\C[\Z_4]$. \emph{Assuming} Conjecture \ref{con_SFIngo} this is the representation category of the vertex algebra $\V_{\SF_n^{even}}$, and thus  

\begin{corollary}
  The representation category of $\W_{B_n,\ell=4}$ is as a modular tensor category in the non-semisimple sense, which equivalent to the representation category of the quasi-Hopf algebra $\tilde{u}$ of type $A_1^n$.
\end{corollary}

We remark, that this special case seems to match our recent more general conjecture \cite{GLO17} for the quantum group side, which would assign to the data $B_n,\ell=4$ a quasi Hopf algebra $U_q^\omega(A_1^n)$ with coradical 
$$U^0=\C[(\Lambda^\olong)^*/\Lambda^\olong]=\C[\Z_4]$$
and with an abelian $3$-cocycle $(\omega,\sigma)$ given by the following quadratic form on $(\Lambda^\olong)^*/\Lambda^\olong$
$$Q(\lambda/\sqrt{2})
=e^{\pi\i(\lambda/\sqrt{2},\lambda/\sqrt{2})}
=q^{(\lambda,\lambda)}$$
which is for a generator $g$ given by $Q(g^k)=e^{\frac{2\pi\i}{8}{k^2}}$.

\section{Other degeneracy cases $C_n,F_4,G_2$ giving extensions of LCFTs}\label{sec_OtherDegeneracies}

In this article we have applied the screening method to the datum $B_n,\ell=4$ to construct a vertex algebra $\W_{B_n,\ell=4}$. Since due to degeneracy only the short screening operators of short roots $A_1^n$ could be used, this vertex algebra is an extension of the following vertex algebra, which is the result of the screening method to the datum $A_1^n,\ell=4$:  
$$\V_{\SF_n^{even}}\cong \W_{B_n,\ell=4} \quad \supset \quad 
\W_{A_1^n,\ell=4}=\left(\W_{A_1,\ell=4}\right)^n\cong \left(\V_{\SF_1^{even}}\right)^n$$
Only the left-hand side contains more module than the latter in its vacuum representation due to the additional long roots, that make additional short roots in the lattice $\Lambda^\olong=\sqrt{2}\Lambda_R^\vee$. More explicitly as $\V_{\Lambda_{A_1^n}^\olong}$-modules
$$\W_{B_n,\ell=4}\;
=\bigoplus_{[\lambda]\in \Lambda_{B_n}^\olong / \Lambda_{A_1^n}^\olong} \V_{[\lambda],A_1^n}$$
$$\Lambda_{B_n}^\olong / \Lambda_{A_1^n}^\olong
=\sqrt{2}\left(\Lambda_{R,B_n}^\vee/\Lambda_{R,A_1^n}\right)=\langle \sqrt{2}\alpha_1^\vee,\ldots \sqrt{2}\alpha_{n-1}^\vee\rangle\cong\Z_2^{n-1}
$$
Hence we can understand the category $\W_{B_n,\ell=4}\md\mod$ from this as an extension as the sense of \cite{CKM17}: We have an algebra $X$ in the category $\W_{A_1^n,\ell=4}$ on the semisimple object
$$X=\bigoplus_{[\lambda]\in \Lambda_{B_n}^\olong / \Lambda_{A_1^n}^\olong} \V_{[\lambda],A_1^n}\qquad \dim(X)=2^{n-1}$$
and the category of local modules over $X$ inside the category $\W_{A_1^n,\ell=4}$ is equivalent to the category $\W_{B_n,\ell=4}$. Thereby only those simple modules $\V_{\lambda+\Lambda_{A_1^n}^\olong}$ are local, which have integral scalar product with $\Lambda_{B_n}^\olong\supset \Lambda_{A_1^n}^\olong$, so for 
$$\lambda\in \left(\Lambda_{B_n}^\olong \right)^*$$
and all $2^{n-1}$ such modules in a coset $\lambda+\Lambda_{B_n}^\olong$ induce up to a single module in $\W_{B_n,\ell=4}$.\\

This way of describing $\W_{B_n,\ell=4}$ as extension of $\W_{A_1^n,\ell=4}\cong \V_{\SF_n^{even}}$ is in some sense opposite to our proof of the vertex algebra isomorphism, where we realized it as the even subspace of a super vertex algebra $\V_{\SF_n}$, which happend to be $\V_{\Lambda_{\B_n,\ell=4}^\oshort}$.\\

The same thing can be done for the other degenerate cases of Lusztig's small quantum group in Lemma \ref{sec_degenerateLusztig}: In each case the (degerate) 
$\W_{\g,\ell}\subset \V_{\Lambda_{\g,\ell}^\olong}$ is the kernel of the short screenings associated to short roots, which form a subset $\g^{(0)}$ of the root system of $\g$ resp. a quotient of the dual root system. Then we can describe $\W_{\g,\ell}$ as local modules over an algebra $X$ in the category. Moreover $U(\g^\vee)^+$, conjecturally even $U(\g^\vee)$,  acts by long screenings as global symmetries. In the following we list the cases and calculate $(\Lambda^\olong)^*/\Lambda^\olong$ from formula (\ref{formula_numberReps}), which counts the number of $\V_{\Lambda^\olong}$-modules. These give rise to irreducible socles for $\W_{\g,\ell}$, which are conjecturally all simple modules. Moreover, by the determinant formula for the index in dual lattices, we easily calculate the dimension of $X$ as the index $\Lambda_{\g}^\olong/\Lambda_{\g^{(0)}}^\olong$ being the squareroot of the ration of these indices. We also calculate the matching central charge:\\

\begin{narrow}{-.8cm}{0cm}
 \begin{tabular}{cc|c|cc|cc}
 $\W_{\g,\ell}$ & $\#$ simples & $\dim X$ & $\W_{\g^{(0)},\ell}$ & $\#$ simples & central charge  & Global symmetry $\g^{(\ell)}$ \\
 \hline
 $\W_{B_n,\ell=4}$ & $2\cdot 2$ & $2^{n-1}$ & $\W_{A_1^n,\ell=2}$ & $2^n\cdot 2^n$ & $-2n$ & $C_n$  \\
 $\W_{C_n,\ell=4}$ & $2\cdot 2^{n-1}$ & $2$ & $\W_{D_n,\ell=2}$ & $4\cdot 2^n$ & $3n^2-2n^3$ & $B_n$ \\
 $\W_{F_4,\ell=4}$ & $1\cdot 2^2$ & $4$ & $\W_{D_4,\ell=2}$ & $4\cdot 2^4$ & $-80$ & $F_4$ \\
 $\W_{G_2,\ell=6}$ & $1\cdot 3$ & $3$ & $\W_{A_2,\ell=6}$ & $3\cdot 3^2$ & $-30$ & $G_2$ \\
\end{tabular}~\\
\end{narrow}

In the first three cases $\V_{\Lambda^\oshort}\supset \V_{\Lambda^\olong}$  are again super-vertex algebras, as are the kernel of the screenings. It is probably helpful to understand these vertex algebras as subalgebras of them.  

\begin{question}
    The first line is the vertex algebra $\V_{\SF_n^{even}}$ of conformal charge $-2n$ extending $(\V_{\SF_1^{even}})^n$ with global symmetry $C_n=\mathfrak{sp}_{2n}$. Do the other three extensions correspond to known logarithmic conformal field theories?  
\end{question}

\begin{acknowledgementX}
We thank I. Runkel and A. Gainutdinov for interesting discussions, very helpful input and improving comments on the final version. Ilaria Flandoli is an ERAMSUS fellow and thankful for this opportunity. We are also grateful for partial support from the DFG Graduiertenkolleg 1670 at the University Hamburg.
\end{acknowledgementX}

\providecommand{\href}[2]{#2}

\end{document}